\documentclass[12pt]{amsart}
\usepackage{amssymb, latexsym, amsthm}
\usepackage{color}
\usepackage{multicol}
\newtheorem{thm}{Theorem}[section] 
\newtheorem*{thm*}{Theorem} 

\newtheorem{cor}[thm]{Corollary}
\newtheorem{defn}[thm]{Definition}

\newtheorem{exmpl}[thm]{Example}
\newtheorem{lem}[thm]{Lemma}
\newtheorem{prop}[thm]{Proposition}

\newtheorem{rem}[thm]{Remark}

\newcommand\Lection[6][]{{\newpage 
\section[#2]{\underline{{{\sc{#2}}\if!#3!{}\else{ (#3)}\fi}}}
\if!#1!{}\else{\tiny{\begin{center}(#1)\end{center}}}\fi\flushright{{\tiny{#5}}} 
\begin{center}{{\bf{#4}}}{\ } \\{#6}\end{center}\vskip1cm}}

\def\ie{{i.e.}} 
\def\eg{{e.g.}}

\def\lam{{\lambda}}

\def\s{\sigma}

\def\C {{\mathbb {C}}}

\def\F {{\mathbb {F}}}
\def\HQ{{\mathbb {H}}} 

\def\R {{\mathbb {R}}}

\def\Z {{\mathbb {Z}}}

\def\AA{{\mathcal A}}

\def\CD{{\mathcal D}}
\def\E{{\mathcal E}}

\def\O{{\mathcal{O}}}

\DeclareMathAlphabet{\mathscr}{OT1}{pzc}{m}{it}

\def\ra{{\rightarrow}}

\def\sub{\subseteq}
\def\({\left(}
\def\){\right)}
\def\isom{{\;\cong\;}} 
\def\normal{{\unlhd}}
\def\semidirect{{\ltimes}} 
\def\normali{{\lhd}} 
\def\co{{\,{:}\,}}

\newcommand\suchthat{{\,:\ \,}}

\newcommand\comp[1]{{#1^{\operatorname{c}}}} 

\newcommand{\oper}[1]{{\operatorname{#1}}}
\newcommand\operA[2]{{\if!#2!\operatorname{#1}\else{\operatorname{#1}_{#2}^{\phantom{I}}}\fi}} 

\newcommand{\Tr}[1][]{\if!#1!\operatorname{Tr}\else{\operatorname{Tr}_{#1}^{\phantom{I}}}\fi} 
\newcommand{\sgn}[2][!]{{\operatorname{sgn}\if!#1(#2)\else\relax\fi}}

\DeclareMathOperator{\Aut}{Aut} %
\newcommand{\Zent}{{\oper{Z}}}

\newcommand\algint[2][]{\if!#1\relax O_{#2}\else{O_{#2}^{\phantom{I}}}\fi}

\renewcommand\H[4][!]{{\operatorname{H}^{#2}\!\!\;({#3},{#4}{\if!#1\relax\else(#1)\fi})}}

\newcommand\cond[2][!]{{\operatorname{cond}_{\if!#1\relax\else{\comp{#1}}\fi}(#2)}}

\renewcommand\L[2]{{\operatorname{L}^{#1}(#2)}}
\newcommand\GL[1][d]{{\operatorname{GL}_{#1}}} 

\newcommand{\set}[1]{{\left\{#1\right\}}}

\newcommand\ideal[1]{{\left<{#1}\right>}}
\newcommand\sg[1]{{\ideal{#1}}}

\newcommand\mul[1]{{#1^{\times}}} 

\newcommand\Pic[1]{{\operatorname{Pic}}} 
\newcommand\Div[1]{{\operatorname{Div}}} 

\newcommand\Cref[1]{{Corollary~\ref{#1}}}

\newcommand\Eref[1]{{Example~\ref{#1}}}
\newcommand\Lref[1]{{Lemma~\ref{#1}}}
\newcommand\Pref[1]{{Proposition~\ref{#1}}}
\newcommand\Rref[1]{{Remark~\ref{#1}}}
\newcommand\Tref[1]{{Theorem~\ref{#1}}}
\newcommand\Fref[1]{{Figure~\ref{#1}}}
\newcommand\Sref[1]{{Section~\ref{#1}}}
\newcommand\Ssref[1]{{Subsection~\ref{#1}}}

\newcommand\eq[1]{{(\ref{#1})}}
\newcommand\eqs[3][--]{{\eq{#2}{#1}\eq{#3}}}
\newcommand\Eq[1]{{Equation \eq{#1}}}

\long\def\half#1\halved{{\footnotesize{#1}}}

\newcommand\thisfile{{\jobname.tex}}
\newcommand\Exref[1]{{Example~\ref{#1}}}

\numberwithin{equation}{section}

\newcommand\ARI[1]{{\left\{\begin{array}{l}\mbox{#1} \end{array}\right\}}}
\newcommand\ARII[2]{{\left\{\begin{array}{ll}\mbox{#1} \\ \mbox{#2}\end{array}\right\}}}
\newcommand\ARIII[3]{{\left\{\begin{array}{lll}\mbox{#1} \\ \mbox{#2} \\ \mbox{#3}\end{array}\right\}}}
\newcommand\ARIV[4]{{\left\{\begin{array}{llll}\mbox{#1} \\ \mbox{#2} \\ \mbox{#3} \\ \mbox{#4} \end{array}\right\}}}
\newcommand\var[2][]{{\mathcal{#2}^{\operatorname{#1}}}}
\def\V{{\var{V}}}
\def\W{{\var{W}}}

\def\MV{{\var{M}}}
\def\ALT{{\var[alt]{A}}}   
\def\FLEX{{\var[flex]{F}}} 
\def\ASS{{\var{A}}}        
\def\INV{{\var[inv]{I}}}      
\def\COMM{{\var{C}}}      
\def\IDEN{{\var{I}^*}}      
\def\DIAS{{\var{D}}}      
\def\Central{{\var{S}^*}}   
\def\superCentral{{\var{S}^*_0}}   
\def\Normal{{\var{N}^*}}   
\def\Exptwo{{\var{E}_2}} 

\def\bs{\backslash}

\newcommand\KZent{{\operatorname{K}}}

\def\perferrow{{\ar@(dl,dr)@{->}|{\circ}}} 

\def\id{{\operatorname{id}}}

\newif\ifXY 
\XYtrue     

\ifXY
\usepackage{xy}
\fi
\ifXY
\xyoption{matrix}
\xyoption{arrow}
\xyoption{curve}
\fi

\newcommand\Nuc[2][]{{\mathfrak{N}_{#1}(#2)}}

\begin{document}

\title[Loops with involution and Cayley-Dickson doubling]{Loops with involution and the Cayley-Dickson doubling process}

\author{Adam Chapman}
\address{School of Computer Science, Academic College of Tel-Aviv-Yaffo, Rabenu Yeruham St., 
Israel}
\email{adam1chapman@yahoo.com}

\author{Ilan Levin}
\author{Uzi Vishne}
\address{Department of Mathematics, Bar-Ilan University, Ramat Gan, Israel
}
\email{ilan7362@gmail.com}
\email{vishne@math.biu.ac.il}

\author{Marco Zaninelli}
\address{
Department of Mathematics, University of Pennsylvania, 
Philadelphia, 
USA}
\email{zaninelli.marco21@gmail.com}

\renewcommand{\subjclassname}{%
      \textup{2010} Mathematics Subject Classification}
\subjclass[2020]{} 


\date{\today}


\begin{abstract}
We develop a theory of loops with involution. On this basis we define a  Cayley-Dickson doubling on loops, and use it to investigate the lattice of varieties of loops with involution, focusing on properties that remain valid in the Cayley-Dickson double. Specializing to central-by-abelian loops with elementary abelian $2$-group quotients, we find conditions under which one can characterize the automorphism groups of iterated Cayley-Dickson doubles. A key result is a corrected proof that for $n>3$, the automorphism group of the Cayley-Dickson loop $Q_n$ is $\GL[3](\F_2) \times \set{\pm 1}^{n-3}$.
\end{abstract}

\thanks{Levin and Vishne are partially supported by an Israeli Science Foundation grant \#1994/20.}

\maketitle

\renewcommand\u[1]{c_{#1}}
\renewcommand\O{\mathbb{O}}
\newcommand\w[1]{\theta_{#1}}
\renewcommand\AA{{\oper{A}\!\!\!\!\Aut(A)}}
\newcommand\IFF{{\Leftrightarrow}}
\newcommand\Nred{{\oper{N}_{\oper{red}}}}

\section{Introduction}

Noncommutative algebra and nonassociative algebra were born in~1843, two months apart, when Hamilton and Graves discovered the Quaternions and the Octonions.
As noticed by Albert \cite{Albert1942}, the procedure invented by Hamilton, that doubles the dimension of a given algebra, can be applied repeatedly, resulting in a sequence of simple nonassociative algebras of dimensions $2^n$. 
From a modern perspective, Hamilton's quaternion algebra belongs to the large family of finite dimensional central simple algebras, whose automorphisms are all inner. On the other hand, the octonions, made public by Cayley, are known to a general mathematical audience mostly for their automorphism group, associated with the simple Lie algebra of type~$G_2$.
Moving further, the automorphism groups of the higher Cayley-Dickson algebras $F = A_0 \subset A_1 \subset \cdots$ were computed by Schafer \cite{Sch0}, also see~\cite{ES}. Namely, for $n > 3$, $\Aut(A_n)$ is a direct product of $\Aut(A_{n-1})$ and an abelian group. This is an intriguing phenomenon: what makes~$A_3$ characteristic in~$A_4$, when~$A_2$ is not characteristic in~$A_3$? (A subalgebra is characteristic if it is stable under any automorphism).

\bigskip

In this paper we explore the loop counterpart of this problem. Let $Q_n$ be the loop composed of the standard basis elements of the classical Cayley-Dickson algebra $A_n$ and their minuses. One can verify by direct computations that $\Aut(Q_3) = \GL[3](\F_2)$. In \cite{Kirsh} it is claimed that $\Aut(Q_n) = \Aut(Q_{n-1}) \times \set{\pm 1}$ for any $n > 3$. This is true, as proven in this paper, but
the proof in \cite{Kirsh} involves false claims (see \Rref{Kirsh}) that we could not work around.

One goal of this paper, therefore, is to provide a correct proof for the claim that $\Aut(Q_n) = \GL[3](\F_2) \times \set{\pm 1}^{n-3}$. But we aim for a more general understanding. The Cayley-Dickson doubling is a celebrated construction of loops, known for producing the smallest Moufang loop \cite{Chein}; also see \cite[Example~II.3.10]{Book8:II}. It is an interesting problem, therefore, to find a reasonably wide setup in which we can compute the automorphism group of an iterated Cayley-Dickson doubling.

\bigskip

The Cayley-Dickson doubling is not a construction of loops, per se, but a construction of loops with involution. In the first part of this paper we therefore develop a basic theory of loops with involution. This is done, in particular, in the context of loops which are central-by-abelian, a class of loops most suitable for commutator and associator calculus. Special involutions, which we term ``central'' and ``normal'', extend the identity involution (which is an involution when $L$ is commutative) and the inverse map (when defined). From another perspective, involutions may be considered a tool in studying loops for their own sake; a loop might have involutions even when the inverse is not well-defined. 
%

\medskip

The second part of the paper initiates a systematic study of the Cayley-Dickson doubling for an arbitrary loop with involution. The process is defined in \Sref{sec:CD}. Commutators and associators in the Cayley-Dickson double are examined in \Sref{sec:cent}, where we also compute the center. In \Sref{sec:cba} we ascertain the conditions on a loop with involution $(L,*)$ under which the Cayley-Dickson double $\CD(L,*)$ is central-by-abelian.

In \Sref{PII} we use the doubling process to investigate the lattice of varieties of loops with involutions (these were not studied before; see \cite{Vars1} for a review of the lattice of loop varieties). In particular, we define the ``derivative'' $\V'$ of a variety $\V$ as the variety of loops whose double is in $\V$; this is a subvariety, and when $\V$ is finitely based, the derivation process always terminates. We then find several examples for ``perfect'' varieties, namely varieties closed under taking the Cayley-Dickson double. The technical notion of homogeneous varieties is defined in \Sref{badsec} as means to get rid of the dependence on parameters of the doubling. The basic examples are given in \Sref{sec:double}.

In \Sref{PIII} we find the conditions for a loop with involution $(L,*)$ under which the Cayley-Dickson double $\CD(L,*)$ is Moufang, extending Chein's approach in \cite{Chein}. For example, when $L$ is a group, $\CD(L,*)$ is Moufang if and only if $*$ is normal (\Cref{fin1}). We also prove that for the $3$-step doubling $\CD^3(L,*;\gamma_1,\gamma_2,\gamma_3)$ to be Moufang, it is necessary and sufficient that $L$ is a commutative Moufang loop and that the involution on~$L$ is the identity (\Cref{triple}).

\medskip

The final part specializes to loops which are central-by-abelian, where we furthermore assume that the quotient $L/\Zent(L)$ is an elementary abelian $2$-group. After observing in \Sref{sec:8} that central and normal involutions coincide in this class, we show in \Sref{sec:dias} that diassociativity, which is not finitely based in general, is indeed finitely based when $L/\Zent(L)$ is as an elementary abelian $2$-group. This leads to a perfect variety of diassociative loops, which are in general not Moufang. Some examples which are Moufang, extending the classical octonion loops, are studied in \Sref{sec:oct}.

In \Sref{sec:stab} we find conditions under which the doubling generator $j$ is fixed under all automorphisms. This finally leads in \Sref{sec:12} to conditions on a loop $T$ under which, when $T \subset L \subset M$ are doubles, $L$ is a characteristic subloop of $M$; hence $\Aut(M)$ is a group extension of $\Aut(L)$ by an abelian group. For our proof to hold, $T$ needs to be a diassociative central extension of an elementary abelian $2$-group, with a central involution, such that elements which are independent modulo the symmetric center do not commute (see \Tref{mainM}). We conclude with a characterization of the automorphism group of an iterated double, see \Tref{FINAL}, from which we obtain that $\Aut(Q_n) = \Aut(Q_3,*) \times \set{\pm 1}^{n-3}$.


\section{Preliminaries}\label{sec:prel}

We collect some basic notions on loops with involution, with an eye towards the main goal which is to define and study loops obtained by iterative doubling. Standard references on loops are \cite{Bruck} and \cite{Book7}.

\subsection{Loops}

A set with a binary operation is a {\bf{loop}} if there is an identity element, and the operation by any element from right or left is invertible. We denote the identity element by $1$. The left and right inverses of $a \in L$ are denoted by $a^{\lam}$ and $a^{\rho}$, respectively, so that $a^{\lam}a = aa^{\rho} = 1$. When the one-sided inverses of $a \in L$ coincide we say that $a$ is invertible and denote the inverse by $a^{-1}$. When $a^{-1}$ is always defined, the loop has a {\bf{well-defined inverse}}.
Following tradition, we sometimes write $ab \cdot c$ for $(ab)c$.

The {\bf{nucleus}} $\Nuc{L}$ of a loop $L$ is the set of elements $a$ satisfying $(ax)y = a(xy)$, $(xa)y = x(ay)$ and $(xy)a = x(ya)$ for any $x,y \in L$. The left-, middle- and right-nucleus are defined by the respective properties, and denoted $\Nuc[\ell]{L}$, $\Nuc[m]{L}$ and $\Nuc[r]{L}$. The nucleus, which is the intersection of the one-sided nuclei, is a subgroup of $L$. If $a \in \Nuc{L}$ then $a$ is invertible and $a^{-1} \in \Nuc{L}$.
The {\bf{commutative center}} $\KZent(L)$ (or {\bf{commutant}}) is the set of elements commuting with all elements of $L$, and the {\bf{center}} is $\Zent(L) = \KZent(L) \cap \Nuc{L}$. The center is an abelian subgroup, but the commutant is not a subloop in general.

\begin{prop}\label{anyloop}
In any loop $L$, the intersection of $\KZent(L)$ with any two of the one-sided nuclei is the center.
\end{prop}
\begin{proof}
For $c \in \KZent(L)$, consider the three elements $(cx)y=(xc)y$, $x(cy)=x(yc)$, and $c(xy)=(xy)c$. Then $c \in \Nuc[\ell]{L}$ when the first and third are equal; $c \in \Nuc[m]{L}$ when the first and second are equal; and $c \in \Nuc[r]{L}$ when the second and third are equal.
\end{proof}


\subsection{Commutators and associators}

The commutators and associators in a loop are the elements $[a,b]$ and $[a,b,c]$ defined by
\begin{equation}\label{commdef}
ab = ba \cdot [a,b] 
\end{equation}
and
\begin{equation}\label{assocdef}
(ab)c = a(bc) \cdot [a,b,c],  
\end{equation}
see \cite[p.~13]{Bruck}. 


Let $[L,L]$ denote the subloop of $L$ generated by all commutators, and $[L,L,L]$ the subloop generated by all associators.
\begin{prop}\label{basicomm1}
If $[a,b] \in \Nuc{L}$, then
\begin{equation}\label{basicomm1-eq}
[b,a] = [a,b]^{-1}.
\end{equation}
\end{prop}
\begin{proof}
By definition we have that $ab = ba \cdot [a,b] = (ab \cdot [b,a])[a,b]$, which by assumption is equal to $ab \cdot [b,a][a,b]$. Cancelling $ab$, we get $[b,a][a,b]=1$, and since $[a,b] \in \Nuc{L}$ this proves the claim.
\end{proof}

\begin{rem}
A loop has the {\bf{right inverse property}} if the identity $(xy)y^{\rho} = x$ holds; this implies all elements are invertible. \Eq{basicomm1-eq} then holds, because $(ab \cdot [b,a])[b,a]^{-1} = ab = ba \cdot [a,b] = (ab \cdot [b,a])[a,b]$.
\end{rem}

\subsection{Loops as central extensions}

A loop $L$ is an {\bf{extension}} of a loop~$A$ by a loop~$K$ if there is a short exact sequence $1 \ra K \ra L \ra A \ra 1$ of loop homomorphisms. The extension is {\bf{central}} if the image of~$K$ is contained in the center of~$L$. We say that~$L$ is {\bf{central-by-abelian}} if~$L$ is a central extension of some abelian {\it{group}}
. The {\bf{associator-commutator subloop}} $L'$, which is by definition the subloop generated by $[L,L]$ and $[L,L,L]$, is always normal, and $L/L'$ is an abelian group. Clearly $L$ is central-by-abelian if and only if $L' \sub \Zent(L)$, that is, if $L/\Zent(L)$ is an abelian group.
We study this class because it is particularly amenable to commutator and associator calculus.


Let $L$ be a central-by-abelian loop.
We use the notation $\xymatrix{x \ar@{->}[r]^k & x'}$ to record that $x' = kx$, when $x,x' \in L$ and $k$ is a central element. For example,
$$\xymatrix{ba \ar@{->}[r]^{[a,b]} & ab}, \qquad \xymatrix{ab \ar@{->}[r]^{[b,a]} & ba}.$$
Similarly, by \eq{assocdef}, we have
$$\xymatrix{a(bc) \ar@{->}[r]^{[a,b,c]} & (ab)c}, \qquad \xymatrix{(ab)c \ar@{->}[r]^{[a,b,c]^{-1}} & a(bc)}.$$

Elements $a,b,c \in L$ can be multiplied in $12$ ways, depicted with the basic factors in \Fref{the12}.

\begin{figure}
$$
\xymatrix@R=24pt@C=24pt{
(ab)c \ar@{<=}[r]^{[a,b,c]} \ar@{<-}[d]_{[a,b]} \ar@{..>}[ddrrr]|(0.21){[c,ab]} & a(bc) \ar@{<..}[ddd]|(0.66){[a,bc]} \ar@{->}[r]^{[c,b]} & a(cb) \ar@{<..}[ddd]|(0.66){[a,cb]} \ar@{=>}[r]^{[a,c,b]} & (ac)b \ar@{->}[d]^{[c,a]} \\
(ba)c \ar@{<=}[d]_{[b,a,c]} \ar@{..>}[ddrrr]|(0.75){[c,ba]} & {} & {} & (ca)b \ar@{<=}[d]^{[c,a,b]} \\
b(ac) \ar@{<..}[rrruu]|(0.75){[b,ac]} \ar@{<-}[d]_{[a,c]} & {} & {} & c(ab) \ar@{->}[d]^{[b,a]} \\
b(ca) \ar@{=>}[r]_{[b,c,a]} \ar@{<..}[rrruu]|(0.2){[b,ca]} & (bc)a \ar@{<-}[r]_{[b,c]} & (cb)a \ar@{<=}[r]_{[c,b,a]} & c(ba).
}
$$
\caption{Products of $a,b,c$ in a central-by-abelian loop.}\label{the12}
\end{figure}
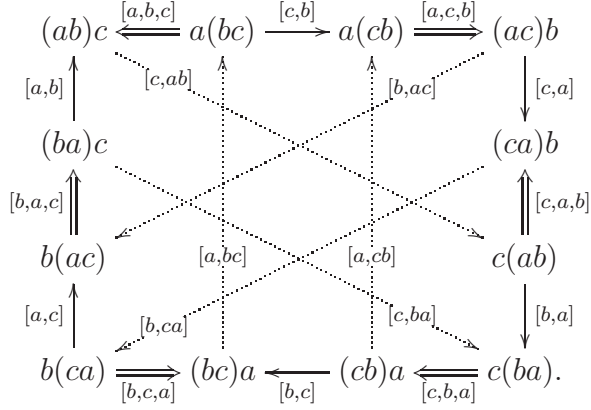

The product of labels over any closed cycle in the diagram is obviously equal to~$1$, so we can read several identities from the diagram.
\begin{cor}
The following identities hold in any central-by-abelian loop:
\begin{eqnarray}
{}[c,ba] & = & [c,ab], \label{M0} \\
{}[a,b,c][c,b,a]  & = & [a,b][a,bc]^{-1} [ab,c][b,c]^{-1}, \label{M1}\\
{}[a,b,c] [b,c,a] [c,a,b] & = & [bc,a][ab,c][ac,b], \label{M2}\\
{}[a,c,b][c,b,a][b,a,c]  & = & [bc,a][ab,c][ac,b]. \label{M3}
\end{eqnarray}
\end{cor}
Of these, \eq{M0} is obvious, because $[a,b]$ is central. The identities \eq{M2} and~\eq{M3} are in the same orbit under permuting the letters $a,b,c$.


\section{Involutions on loops}\label{sec:inv}

An {\bf{involution}} on $L$ is an anti-automorphism of order~$2$. We denote an involution by~$*$. For any involution, it is easy to see that $(a^*)^\lam = (a^\rho)^*$ and $(a^*)^\rho = (a^\lam)^*$ for every $a \in L$. An involution must preserve the nucleus and the center. Furthermore, the identity map is an involution if and only if $L$ is commutative.

We are not aware of a systematic treatment of loops with involution. For a survey of varieties of semigroups with involutions, see \cite{Surv}.

We focus on some special types of involutions.  We will refer below to the following small example.
\begin{exmpl}\label{-1ex}
Let $L = \set{\pm 1, \pm \s, \pm \s^2}$ be the loop of order~$6$ generated by a central element of order~$2$ denoted $-1$, and an element $\s$, satisfying
$$\s \cdot \s = \s^2, \qquad \s \cdot \s^2 = 1, \qquad \s^2 \cdot \s = -1, \qquad \s^2 \cdot \s^2 = \s.$$
Then $L$ is a central extension of the cyclic group~$\Z_3$ by $\set{\pm 1} \isom \Z_2$. The center is $\set{\pm 1}$.
This loop has an involution defined by $\s^* = \s^2$ and $(\s^2)^* = \s$, where clearly $(-1)^* = -1$.
\end{exmpl}

\subsection{Central involutions}


We say that an involution $*$ on a loop $L$ is {\bf{central}} if for every $a \in L$ there exists $\mu(a)$ in the center of~$L$, such that
\begin{equation}\label{mudef}
a^* = \mu(a)a.
\end{equation}
Equivalently,
\begin{equation*}
\mu(a) = a^*a^{\rho} = a^{\lam} a^* \in \Zent(L).
\end{equation*}
Whenever we refer to a central involution, the associated function~$\mu$ is tacitly assumed to be defined by \eq{mudef}.
If $*$ is central, then
\begin{equation}\label{aa^*}
aa^* = a^*a.
\end{equation}

\begin{rem}\label{L/Z-*1}
An involution on $L$ is central if and only if it induces the identity involution on $L/\Zent(L)$.
Therefore, a necessary condition for the existence of a central involution is that $L/\Zent(L)$ is a commutative loop, equivalently $[L,L] \sub \Zent(L)$.
\end{rem}

This condition does not suffice. There are central-by-abelian loops without central involutions:
\begin{exmpl}[A central-by-abelian loop with involution which does not have a central involution]\label{0ex}
The loop $L$ of \Exref{-1ex} has an involution, but has no central involution. For if $x\mapsto x^{\tau}$ is a central involution then $\s^\tau = \pm \s$ implies $(\s^2)^\tau = (\s^\tau)^2 = (\pm \s)^2 = \s^2$, and then $\s^{\tau} = (\s^2 \cdot \s^2)^{\tau} = ((\s^2)^{\tau})^2 = (\s^2)^2 = \s$; but now $1 = 1^\tau = (\s  \s^2)^{\tau} = (\s^2)^\tau \s^\tau = \s^2 \s = -1$, a contradiction.
\end{exmpl}

\begin{prop}
Let $*$ be a central involution on a loop $L$. Then
\begin{eqnarray}
\mu(a^*) & = & \mu(a)^{-1}.\label{theyantisymm1} \\
\mu(a)^* & = & \mu(a)^{-1}.\label{theyantisymm2}
\end{eqnarray}
for any $a \in L$.
\end{prop}
\begin{proof}
Compute $a = a^{**} =\mu(a^*)a^* = \mu(a^*)\mu(a)a$
and $a = a^{**} = (\mu(a)a)^* = a^* \mu(a)^* = \mu(a) a \mu(a)^*$.
\end{proof}

\begin{prop}\label{muinduces}
Let $*$ be a central involution on a loop $L$. Then
\begin{equation}\label{goodact}
\mu(\alpha a) = \mu(\alpha) \mu(a),
\end{equation}
for any $\alpha \in \Zent(L)$ and $a \in L$. In particular $\mu \co L \ra \Zent(L)$ restricts to a group homomorphism $\mu \co \Zent(L) \ra \Zent(L)$.
\end{prop}
\begin{proof}
Compute $\mu(\alpha a) \alpha a = (\alpha a)^* = a^* \alpha^* = \mu(\alpha) \mu(a) \alpha a$.
\end{proof}

\begin{rem}
For any $a,b$ we have that
$\mu(a)\mu(b)ba = b^*a^* = (ab)^* = \mu(ab) ab$, showing that
\begin{equation}\label{niceobs}
{}[a,b] = \mu(a)\mu(b) \mu(ab)^{-1}.
\end{equation}
\end{rem}


\begin{rem}
A function $\mu \co L \ra \Zent(L)$ induces an involution on $L$ by $a^* = \mu(a)a$ if and only if it satisfies \eq{theyantisymm1} and \eq{niceobs}.
\end{rem}

\begin{prop}\label{basicomm2}
If $[L,L] \sub \KZent(L)$ and $*$ is any involution,
then
\begin{equation}\label{commutator*}
[a,b]^* = [b^*,a^*]
\end{equation}
for any $a,b \in L$.
\end{prop}
\begin{proof}
From \eq{commdef} (twice) we get
$a^*b^*\cdot [b^*,a^*] = b^*a^* = (ab)^* = (ba\cdot [a,b])^* = [a,b]^* \cdot a^*b^*$, but the commutators commute with $a^*b^*$.
\end{proof}

\subsection{Super-central involutions}

An element~$a$ satisfying $a^* = a$ is called {\bf{symmetric}}. We need a condition for the commutators to be symmetric under the involution. Let $\Zent(L,*)$ denote the symmetric part of the center.

\begin{prop}\label{sc}
Let $L$ be a loop with a central involution $x^* = \mu(x)x$.
\begin{enumerate}
\item The following are equivalent:
\begin{enumerate}
\item[(a)] $\mu(a)^* = \mu(a)$ for all $a$;
\item[(b)] $\mu(a)^2 = 1$ for all $a$;
\item[(c)] every $a^2$ is symmetric.
\end{enumerate}
\item \label{sc2}
The following are equivalent:
\begin{enumerate}
\item[(d)] $[L,L] \sub \Zent(L,*)$;
\item[(e)] $[a,b]^2 = 1$ for every $a,b$.
\end{enumerate}
\item The conditions (a)--(c) imply the conditions (d),(e).
\end{enumerate}
\end{prop}
\begin{proof}
That (a),(b) are equivalent follows from \eq{theyantisymm2}. By \eq{niceobs} we have that $\mu(a^2) = \mu(a)^2$, so (b) is equivalent to $(a^2)^* = a^2$. Since the involution is central, and using \eq{commutator*} and \eq{basicomm1-eq}, we have that $[a,b]^* = [b^*,a^*] = [\mu(b)b,\mu(a)a] = [b,a] = [a,b]^{-1}$; so (d) and (e) are equivalent.
The third statement follows from \eq{niceobs}.
\end{proof}

A central involution satisfying the conditions (a)--(c) of \Pref{sc} is called {\bf{super-central}}. Central and super-central involutions appear naturally when studying the Cayley-Dickson double.

\subsection{Normal involutions}

We present another type of involution, which naturally appear when we study Moufang loops in \Sref{PIII}.

We say that an involution~$*$ on~$L$ is {\bf{normal}} if $\nu(x) = x^*x$ is in the center for every $x \in L$. The inverse map, for loops in which inverse is well-defined, is a key example. As before, when $*$ is a normal involution we always denote by $\nu \co L \ra \Zent(L)$ the associated map defined by $\nu(x) = x^*x$.

\begin{rem}\label{normalInv}
Let $*$ be a normal involution. Then $(\nu(x)^{-1}x^*)x = 1 = x(x^*\nu(x^*)^{-1})$ shows that
\begin{equation}\label{starnormal}
x^* = \nu(x)x^\lam = \nu(x^*)x^{\rho}.
\end{equation}
Moreover $[x,x^*]=1$ for every $x$ if and only if $\nu(x^*)=\nu(x)$ for every $x$, if and only if the inverse is well-defined.
\end{rem}

When the center is trivial, the inverse is the only normal involution. But normal involutions may exist also when the inverse is not necessarily defined.
\begin{exmpl}[A loop with normal involution where inverses are not defined]\label{1ex}
The involution in \Exref{-1ex} is normal, with $\nu(\s) = -1$ and $\nu(\s^2) = 1$. But the right inverse $\s^{\rho} = \s^2$ differs from the left inverse $\s^{\lam} = -\s^2$.
\end{exmpl}

\begin{rem}
For any normal involution we always have \begin{equation}\nu(x^*) = \nu(x^{\rho})^{-1}.\end{equation}
(Indeed $x = x^{**} = \nu(x^*)x^{*\lam} = \nu(x^*) x^{\rho*} = \nu(x^*)\nu(x^{\rho})x^{\rho\lam} = \nu(x^*)\nu(x^{\rho})x$.)
\end{rem}

Suppose the involution is normal. Modulo the center, the left- and right- inverses coincide. This proves that the involution induces the inverse operation on $L/\Zent(L)$, which therefore has to satisfy the anti-automorphic inverse property. (Compare this necessary condition to \Rref{L/Z-*1}).

It is easy to see that the associated map $\nu$ satisfies $\nu(xy) = \nu(x)\nu(y)$ if $x$ or $y$ is in the nucleus. Thus, a normal involution always induces a homomorphism $\nu \co \Nuc{L} \ra \Zent(L,*)$. Recall that a loop satisfies the {\bf{anti-automorphic inverse property}} if $(xy)^\lam = y^\lam x^\lam$ (equivalently $(xy)^\rho = y^\rho x^\rho$).
\begin{prop}\label{normalhom}
Let $L$ be a loop with normal involution $*$. Then $L$ has the anti-automorphic inverse property if and only if $\nu \co L \ra \Zent(L,*)$ is a homomorphism.
\end{prop}
\begin{proof}
Using \eq{starnormal} we have that $\nu(xy)(xy)^{\lam} = \nu(x)\nu(y) y^{\lam}x^{\lam}$.
\end{proof}

Under the anti-automorphic inverse property the map $c \mapsto cc^*$ is a ``norm'' homomorphism to the center, hence the term ``normal''; also see \Rref{Shtut}.

\begin{rem}
In \cite{Chein+}, Chein constructed extensions of Moufang loops by $\Z_{2k}$ which are Moufang; this uses an ``appropriately chosen'' map $\theta$ (see \cite[Example~II.3.10]{Book8:II}). A normal involution satisfies this condition.
\end{rem}

\subsection{Central-by-abelian loops with involution}

Let $L$ be a central-by-abelian loop with an involution $*$.
Similarly to \eq{commutator*}, since now associators are central, we have
\begin{equation}\label{associator*}
[c^*,b^*,a^*]^{-1} = [a,b,c]^*.
\end{equation}
Indeed,
$[a,b,c]^* \cdot ( c^*(b^*a^*) \cdot [c^*,b^*,a^*] ) =
[a,b,c]^* \cdot (c^*b^*)a^* = (a(bc)\cdot[a,b,c])^* = ((ab)c)^* = c^*(b^*a^*)$, and $[a,b,c]^* \in \Zent(L)$.


\begin{prop}
Let $L$ be a central-by-abelian loop with a central involution $*$.  Then
$$[a,b,c] \cdot [c,b,a] = \mu([a,b,c])^{-1}.$$
\end{prop}
\begin{proof}
Since $x^* = \mu(x)$ and $\mu(x)$ is central for every $x$, we have that
$$\mu((ab)c)\,(ab)c = ((ab)c)^* = c^*(ab)^* = c^*(b^*a^*) = \mu(a)\mu(b)\mu(c)\, c(ba);$$
$$\mu(a(bc))\,a(bc) = (a(bc))^* = (bc)^*a^* = (c^*b^*)a^* = \mu(a)\mu(b)\mu(c)\, (cb)a.$$
Since $(ab)c =  a(bc) \cdot [a,b,c]$ and $(cb)a = c(ba) \cdot [c,b,a]$, we conclude that
$[a,b,c][c,b,a] = \mu(a(bc))\mu((ab)c)^{-1}$, which is equal to $\mu([a,b,c])^{-1}$ by applying \eq{goodact} to \eq{assocdef}.
%
\end{proof}

Analogously to \Pref{sc}, we have:
\begin{prop}\label{sc3}
Let $L$ be a central-by-abelian loop with a central involution $x^* = \mu(x)x$.
The following are equivalent:
\begin{enumerate}
\item[(d$'$)] $[L,L,L] \sub \Zent(L,*)$;
\item[(e$'$)] $[a,b,c][c,b,a] = 1$ for every $a,b,c$.
\end{enumerate}
\end{prop}
\begin{proof}
By centrality we have that $[c^*,b^*,a^*] = [c,b,a]$. Now by \eq{associator*} we have that $[a,b,c]^*[c^*,b^*,a^*] = [a,b,c]^*[c,b,a] = 1$, so $[a,b,c]$ is symmetric if and only if $[a,b,c][c,b,a] = 1$.
\end{proof}

\section{Cayley-Dickson doubling}\label{sec:CD}

Composition algebras are obtained from the base field by a repeated use of the Cayley-Dickson doubling process for algebras.
In this section we describe such a process for loops.

\subsection{The doubling process}

Let $(L,*)$ be a loop with involution.  Let~$\gamma$ be a central element of~$L$.
The {\bf{Cayley-Dickson double}} of $L$ is the loop $\CD(L,*,\gamma) = L \cup Lj$, whose multiplication table is defined by
\begin{eqnarray}
a(bj) & = & (ba)j \label{Mul1}  \\ 
(aj)b & = & (ab^*)j \\ 
(aj)(bj) & = & \gamma b^*a \label{Mul3}
\end{eqnarray}
for any $a,b\in L$. In particular $jb = b^*j$, $j^2 = \gamma$, and $[j,j,j] = \gamma^*\gamma^{-1}$.

\begin{rem}
Our definition is compatible with positioning the new generator~$j$ to the right, letting it act by $jb = b^*j$. In most sources (e.g.\
\cite{Sch}; \cite[p.~458]{BOI}; \cite[Section~6.3]{Conway}; and \cite[p.~28]{ZSSS}) the multiplication scheme is $a(bj)=(a^*b)j$; $(aj)b = (ba)j$; and $(aj)(bj) = \gamma ba^*$. One would obtain our scheme by applying this scheme to the opposite loop, then taking the opposite. Alternative schemes are compared and studied in \cite{KPVonallCD}, which, following \cite{Vvar}, shows that the construction we use here (equivalently its conjugate under the opposite) best preserves properties of the double.
\end{rem}

\begin{prop}\label{Disloop}
The magma $\CD(L,*,\gamma)$ is indeed a loop with respect to the operations defined in \eq{Mul1}--\eq{Mul3}.
\end{prop}
\begin{proof}
Verify that the left- and right- inverse functions extend from $L$ to $\CD(L,\gamma)$ as
\begin{equation}\label{gam+rho}
(aj)^\lam = (\gamma^{-1} (a^\lam)^*)j, \qquad (aj)^\rho = (\gamma^{-*} (a^\lam)^*)j.
\end{equation}
It is then easy to check that multiplication by $a \in L$ or $aj$ from left or right is invertible.
\end{proof}
\begin{cor}\label{Inv-perfect}
The inverse of elements in $Lj$ is well-defined if and only if $\gamma^* = \gamma$. Therefore, $\CD(L,*,\gamma)$ has well-defined inverses if and only if $L$ has well-defined inverses and $\gamma^* = \gamma$.
\end{cor}

Here is the most famous instance of this loop construction:
\begin{exmpl}\label{onChein}
Let $G$ be a group. Then the double $\CD(G,{\,}^{-1},1)$, usually denoted $M(G,2)$, is a Moufang loop, see \cite{Chein}. Chein furthermore proved that a Moufang loop whose every minimal set of generators has an element of order~$2$ is of this form, and systematically utilized this construction in \cite{Chein+}.
\end{exmpl}

The interplay of the doubling construction with subloops and quotients is rather clear. Recall that a subloop of $L$ is {\bf{normal}} if it is the kernel of a homomorphism from $L$.
\begin{lem}
Let $L$ be a loop with involution. Let $H \leq L$ be a subloop closed under the involution.
\begin{enumerate}
\item The subloop $\sg{H,j}$ of $\,\CD(L,*,\gamma)$ is isomorphic to $\CD(\sg{H,\gamma},*,\gamma)$.
\item
If $N$ is normal in $L$, then $N$ is normal in $\CD(L,*,\gamma)$, and
$$\CD(L,*,\gamma)/N \isom \CD(L/N,*,\bar{\gamma})$$
where $\bar{\gamma}$ denotes the image of $\gamma$ in $L/N$.
\end{enumerate}
\end{lem}
Another observation, which uses the normality $L \normal \CD(L)$, is concerned with finitely generated subloops:
\begin{rem}\label{geni}
If a subloop of $\CD(L)$ which is not contained in~$L$ can be generated by $k$ elements, then it can be written in the form $\sg{a_1,\dots,a_{k-1},bj}$ for $a_1,\dots,a_{k-1},b \in L$ (replace $\sg{aj,bj}$ by $\sg{(aj)(bj),bj}$ when necessary).
\end{rem}

\begin{rem}
Recall that a finite loop is {\bf{Lagrange}} is the order of every subloop divides the order of the loop; and {\bf{sub-Lagrange}} if every subloop is Lagrange. By \cite[Lemma~V.2.1]{Bruck}, both properties pass from $L$ to its Cayley-Dickson double.
\end{rem}

\subsection{Extending the involution}

Let $(L,*)$ be a loop with involution and $M = \CD(L,*,\gamma)$ its Cayley-Dickson double with respect to some central element $\gamma$.
\begin{rem}
If $(aj)^* = \E(a)j$ extends the involution to $M$ for some bijective map $\E \co L \ra L$, then necessarily $\E(a) = ea$ for $e = \E(1)$.
A priori, $(aj)^* = (ea)j$ extends the involution if and only if $e^2 = 1$, $(ea)b = e(ab)$ and $\gamma (ea)^*(eb) = \gamma^* a^*b$. As detailed below, these conditions hold, for example, if $e \in \Nuc{L}$, $e^2 = 1$, and $\gamma e$ is symmetric.
%
\end{rem}
In the interest of simplicity, we will assume from now on that $(aj)^* = (\epsilon a)j$ for~$\epsilon$ a \emph{central} element.
\begin{prop}\label{basiccond}
Let~$\epsilon$ be a central element of $L$.
The involution $*$ extends from $L$ to $\CD(L,\gamma)$ by
\begin{equation}\label{extinv}
(a j)^* = (\epsilon a) j \qquad  (a \in L),
\end{equation}
if and only if
\begin{equation}\label{thecond}
\epsilon^2 = 1, \qquad (\epsilon\gamma)^* = \epsilon \gamma.
\end{equation}
\end{prop}
\begin{proof}
First assume $(aj)^* = (\epsilon a)j$ extends the involution. Then $j = j^{**} = (\epsilon j)^* = (\epsilon^2) j$, so $\epsilon^2 = 1$. Also, applying the involution to $j^2 = \gamma$ we get $\gamma^* = (j^2)^* = (j^*)^2 = (\epsilon j)(\epsilon j) = \gamma \epsilon^*\epsilon$, so that $\epsilon^* \gamma^* = (\epsilon \gamma)^* = \gamma^* \epsilon^{-*} = \epsilon \gamma$.

On the other hand, if $\epsilon^2 = 1$ and $\epsilon \gamma$ is symmetric, one can verify that $x^{**} = x$ and $(xy)^*= y^* x^*$ for every $x,y \in \CD(L,\gamma)$ using the defining relations.
\end{proof}

The data defining the Cayley-Dickson double always includes the loop~$L$, the involution~$*$, and central elements~$\gamma,\epsilon$ satisfying the condition \eq{thecond}. As permitted by the context, we will interchangeably denote the double by $\CD(L)$, $\CD(L,*)$, $\CD(L,*,\gamma)$ or $\CD(L,*,\gamma,\epsilon)$. The following observations will be used frequently.
\begin{rem}\label{ongamma}
\begin{enumerate}
\item
$\gamma^* = \gamma$ if and only if $\epsilon^* = \epsilon$.
\item\label{ongamma2} $\gamma^{*}\gamma^{-1} = \epsilon^{*}\epsilon^{-1} = \epsilon^*\epsilon$ is symmetric.
\end{enumerate}
\end{rem}
Unless stated otherwise, we assume that $\gamma, \epsilon$ are symmetric.

\begin{exmpl}\label{trivi}
If $L$ is commutative with the identity involution, then
$$\CD(L,*,1,1) \isom (L\times \Z_2,\ * \times \id)$$
(as loops with involution).
\end{exmpl}

\begin{prop}\label{N'0}
\begin{enumerate}
\item
The extension of a normal involution $*$ on $L$ to $M = \CD(L,*,\gamma,\epsilon)$ is always normal.
\item
The involution on $M$ is the inverse map if and only if the involution on $L$ is the inverse map and $\gamma = \epsilon = \epsilon^*$.
\end{enumerate}
\end{prop}
\begin{proof}
If $*$ is normal on $L$, then $(aj)^*(aj) = \gamma \epsilon a^*a$ is in $\Zent(L,*) \sub \Zent(M)$ for any $a \in L$, by~\eq{thecond}. The second claim follows from the same computation, noting also that $(aj)(aj)^* = \gamma\epsilon^*a^*a$.
\end{proof}

\subsection{Some automorphisms of the double}\label{ss:Aut0}

Let $(L,*)$ be a loop with involution. We denote by $\Aut(L,*) \sub \Aut(L)$ the subgroup composed of the automorphisms of $L$ respecting the involution, namely those $\s \in \Aut(L)$ satisfying $\s(x^*) = \s(x)^*$.

Let $\gamma,\epsilon \in \Zent(L)$ be elements satisfying \eq{thecond}. Let $Z = \Zent(L)$ denote the center. For $M = \CD(L,*)$, we denote by $\Aut(M;\,L,Zj)$ the group of automorphisms of $M$ preserving $L$ and $Zj$ (as sets), and by $\Aut(M,*;\,L,Zj)$ its subgroup of automorphisms commuting with the involution on $M$.

The natural action of $\Aut(L)$ on $\Zent(L)$ defines a pair of semidirect products of groups,
$$\Aut(L,*) \semidirect \Zent(L) \sub \Aut(L) \semidirect \Zent(L).$$
\begin{prop}\label{Aut0}
We have that $\Aut(M;\,L,Zj) \leq \Aut(L,*) \semidirect \Zent(L)$. More explicitly,
$\s \mapsto (\s|_L, \s(j)j^{-1})$ induces group isomorphisms
\begin{enumerate}
\item $\Aut(M;\,L,Zj) \ra \set{(\s,p) \in \Aut(L,*) \semidirect \Zent(L) \suchthat
\s(\gamma)\gamma^{-1} = p^*p}$;
\item $\Aut(M,*;\,L,Zj) \ra \set{(\s,p) \in \Aut(L,*,\epsilon) \semidirect \Zent(L) \suchthat
\s(\gamma)\gamma^{-1} = p^*p}$,
\end{enumerate}
where $\Aut(L,*,\epsilon)$ is the subgroup consisting of automorphisms fixing~$\epsilon$.
\end{prop}
\begin{proof}
Every $\s \in \Aut(M;\,L,Zj)$ induces an automorphism $\s|_L \in \Aut(L)$, and specifies an element $p = \s(j)j^{-1} \in Z$. Then $\s(\gamma) = \s(j^2) = \s(j)^2 = (pj)^2 = \gamma p^*p$, so $\s(\gamma)\gamma^{-1} = p^*p$. Furthermore, $(p\s(a)^*)j = (pj)\s(a) = \s(j)\s(a) = \s(ja) = \s(a^*j) = \s(a^*)\s(j) = \s(a^*)(pj) = (p\s(a^*))j$ proves that $\s$ must preserve the involution on $L$. This shows that the map in (1) is well-defined. It is then a matter of computation to show that every choice of $\s$ and $p$ satisfying the condition, defines an automorphism of $M$.

For (2), observe that for $\s \in \Aut(M;\,L,Zj)$ to preserve the involution on $M$ it is necessary and sufficient that
$\s((aj)^*) = \s((\epsilon a)j) = \s(\epsilon a)(pj) = (p\s(\epsilon a))j$ and
$\s(aj)^* = (\s(a)(pj))^* = ((p\s(a))j)^* = (\epsilon p \s(a))j$ are equal; namely $\s(\epsilon) = \epsilon$.
\end{proof}

\section{The center of the double}\label{sec:cent}

Let $(L,*)$ be a loop with involution. In this section we compute the center of a Cayley-Dickson double of $(L,*)$.

\subsection{The commutative center}
We first compute the commutative center of $\CD(L,*)$.

\begin{prop}\label{dercomm}
$\CD(L,*,\gamma)$ is commutative if and only if $*$ is the identity on~$L$.
\end{prop}
\begin{proof}
For $\CD(L,*,\gamma)$ to be commutative we need $L$ to be commutative and
$aj = ja =a^*j$ so $a^* = a$ for every $a \in L$. On the other hand, if $*$ is the identity then $a(bj) = (ba)j = (ba^*)j = (bj)a$ and $(aj)(bj) = \gamma b^*a = \gamma a^*b = (bj)(aj)$ for every $a,b \in L$, taking into account that $L$ must be commutative.
\end{proof}

Recall that $\Zent(L,*)$ is the set of symmetric elements in the center of~$L$. Similarly we denote by $\KZent(L,*)$ the set of symmetric elements in the commutative center $\KZent(L)$. We use the term {\bf{centralizer}} of a set $A$ in $L$ to denote the elements of $L$ commuting with every element of $A$ (regardless of the lack of associativity).
\begin{prop}\label{thecenter}
Assume $*$ is nonidentity on $L$. Then
\begin{enumerate}
\item The centralizer of $L$ in $\CD(L,*)$ is $\KZent(L)$.
\item The centralizer of $Lj$ in $\CD(L,*)$ is $\set{x \in L \suchthat x^* = x}$. 
\item The commutative center of $\CD(L,*)$ is $\KZent(\CD(L,*)) = \KZent(L,*)$.
\end{enumerate}
\end{prop}
\begin{proof}
This is an easy computation.\end{proof}

\subsection{The nucleus and the center}\label{ss:n+c}

Let us now compute the nucleus of $\CD(L,\gamma)$. We first compute, once and for all, the products of three elements of each of the two types. For $a,b,c \in L$, we have that:
\begin{equation}\label{TAB}
\begin{array}{ccc||cc}
x & y& z & (xy)z & x(yz) \\
\hline
a&b&c & (ab)c & a(bc)  \\
a&b&cj & (c(ab))j & ((cb)a)j \\
a&bj&c & ((ba)c^*)j & ((bc^*)a)j\\
a&bj&cj & \gamma c^*(ba) & \gamma a(c^* b) \\
aj&b&c & ((ab^*)c^*)j & (a(c^*b^*))j  \\
aj&b&cj & \gamma c^*(ab^*) & \gamma (b^*c^*) a \\
aj&bj&c & \gamma (b^*a)c & \gamma (cb^*)a \\
aj&bj&cj &  (\gamma\,c(b^*a))j & (\gamma^*\, a(b^*c))j
\end{array}
\end{equation}

Before we compute the one-sided nuclei of $\CD(L,*)$, we note the following general fact:
\begin{rem}\label{N123}
In any loop with involution $(L,*)$,
\begin{enumerate}
\item $a \in \Nuc[\ell]{L} \Leftrightarrow a^* \in \Nuc[r]{L}$,
\item $a \in \Nuc[m]{L} \Leftrightarrow a^* \in \Nuc[m]{L}$. 
\end{enumerate}
\end{rem}

\begin{prop}\label{nucM}
The intersection with $L$ of the nucleus and one-sided nuclei of $\CD(L,*)$ are as follows.
\begin{enumerate}
\item $\Nuc[\ell]{M} \cap L = \Nuc[r]{M} \cap L = \Zent(L)$.
\item $\Nuc[m]{M} \cap L = \Nuc[m]{L} \cap \KZent(L)$.
\item $\Nuc{M} \cap L = \Zent(L)$.
\end{enumerate}
\end{prop}
\begin{proof}
We extract the conditions from Table \eq{TAB}. Let $a \in L$. Then $a \in \Nuc[\ell]{M}$ if and only if $(xy)z=x(yz)$ in lines $1$--$4$ of the table, verifying the claim for $x = a$ and any $y,z \in L \cup Lj$. The conditions are $(ab)c=a(bc)$, $c(ab)=(cb)a$, $(ba)c = (bc)a$ and $c(ba)=a(cb)$, for any $b,c \in L$. Taking $c = 1$ in the second condition shows that $a \in \KZent(L)$. The first and third conditions are that $a \in \Nuc[\ell]{L}$, and the second and fourth are that $a \in \Nuc[r]{L}$. By \Rref{anyloop}, this proves $\Nuc[\ell]{M} \cap L = \Zent(L)$, and $\Nuc[r]{M} \cap L = \Zent(L)$ follows from \Rref{N123}.

Next, for $b \in L$, $b \in \Nuc[m]{M}$ if and only if $(xy)z=x(yz)$ in lines 1,2,5 and 6 of the table, for $y = b$ and any $a,c \in L$. Namely,
$(ab)c=a(bc)$, $c(ab)=(cb)a$, $(ab)c = a(cb)$ and $c(ab) = (bc)a$. Again, taking $c = 1$ in the second condition shows that $b \in \KZent(L)$. With this, all the conditions become equivalent to $b \in \Nuc[m]{L}$.
\end{proof}

We now compute the intersection of the one-sided nuclei with $Lj$.
\begin{prop}\label{nucMall}
Let $M = \CD(L,*,\gamma)$.
\begin{enumerate}
\item
\begin{enumerate}
\item If $\gamma^* = \gamma$,
\begin{eqnarray*}
\Nuc[\ell]{M} \cap Lj & = & \Nuc[r]{M} \cap Lj \\
 & = & \set{aj \suchthat a \in \KZent(L), (\forall b,c) (ab)c = a(cb), c(ba) = (bc)a}, \\
\Nuc[m]{M} \cap Lj & = & \set{aj \suchthat a \in \KZent(L), (\forall b,c)    (ab)c = (ac)b, c(ba) = b(ca)};
\end{eqnarray*}
\item Otherwise,
\begin{eqnarray*}
\Nuc[\ell]{M} = \Nuc[r]{M} = \Zent(L), \\
\Nuc[m]{M} = \Nuc[m]{L} \cap \KZent(L).
\end{eqnarray*}
\end{enumerate}
\item The nucleus is
$$\Nuc{\CD(L,\gamma)} = \begin{cases} \Zent(L)\cup \Zent(L)j & \mbox{if $L$ is commutative and $\gamma^*=\gamma$}, \\
\Zent(L) & \mbox{otherwise}.\end{cases}.$$
\end{enumerate}
\end{prop}
\begin{proof}
As in the proof of \Pref{nucM}, the conditions for $aj \in \Nuc[\ell]{M}$ or $aj \in \Nuc[m]{M}$ are extracted from Table~\eq{TAB}. If $\gamma^* \neq \gamma$ then the final row shows that the intersection of any one-sided nucleus with $Lj$ is empty.

Assume $\gamma^* = \gamma$. Then (by the final four rows), $aj \in \Nuc[\ell]{M}$ if and only if $(ab)c = a(cb)$, $c(ab) = (bc)a$, $(ba)c = (cb)a$ and $c(ba) = a(bc)$. Taking $c = 1$ in the second condition, we see that $a \in \KZent(L)$, and the other conditions become $(ab)c = a(cb)$ and $c(ba) = (bc)a$. Applying the involution via \Rref{N123}, we also obtain $\Nuc[r]{M} \cap Lj$. Likewise, by rows 3,4,7 and 8 (where we apply the involution on the latter two, to obtain a condition on~$b$), $bj \in \Nuc[m]{M}$ if and only if $(ba)c = (bc)a$, $c(ba) = a(cb)$, $c(ab) = a(bc)$, and $(ab)c = (cb)a$. Taking $c = 1$ in the second condition shows that $b \in \KZent(L)$, and the conditions become $(ba)c = (bc)a$ and $c(ab) = a(cb)$. The stated intersection is obtained by switching $a,b$.

Finally, $aj \in \Nuc[\ell]{M} \cap \Nuc[m]{M}$ if and only if $a \in \KZent(L) \cap \Nuc[\ell]{L} \cap \Nuc[r]{L} = \Zent(L)$, but it is also necessary that $bc = cb$ for any $b,c \in L$.
\end{proof}

\begin{cor}\label{ZentMnot1}
Assume $*$ is nonidentity. Then $\Zent(\CD(L,*)) = \Zent(L,*)$.
\end{cor}
\begin{proof}
Denote $M = \CD(L,*)$. We proved above that $\Zent(M) = \KZent(M) \cap \Nuc{M} = \Zent(L,*) \cap \Nuc{M} = \Zent(L,*) \cap (L \cap \Nuc{M}) = \Zent(L,*) \cap \Zent(L) = \Zent(L,*)$.
\end{proof}

To complete the picture, we compute the center when $*$ is the identity map on $L$.
\begin{prop}\label{ZentM1}
Assume $*$ is the identity. Then $\Zent(\CD(L,\gamma)) = \CD(\Zent(L),\gamma)$.
\end{prop}
\begin{proof}
Write $M = \CD(L,\id,\gamma)$. Since $M$ is commutative, the center coincides with the nucleus. In \Pref{nucM} we compute that $\Zent(M) \cap L = \Zent(L)$. Likewise if $aj \in \Zent(M) \cap Lj$ then $a \in \Nuc{L}$ by imposing the conditions $[aj,b,c] = [b,aj,c] = [b,c,aj] = 1$. The reverse direction is clear (or again use the table). We get $\Zent(M) = \Zent(L)\cup \Zent(L)j$ which is the double of $\Zent(L)$.
\end{proof}

In either case we have for $M = \CD(L,*,\gamma)$ that
\begin{equation}\label{theZent}
\Zent(M) \cap L = \Zent(L,*).
\end{equation}

It follows immediately from \eq{theZent} that
\begin{equation}\label{thesymmZent}
\Zent(M,*) \cap L = \Zent(L,*)
\end{equation}
as well. In particular:
\begin{prop}\label{keepthat}
The involution on $M = \CD(L,*,\gamma,\epsilon)$ is central if and only if
the involution on $L$ is super-central and $\gamma^* = \gamma$.
\end{prop}
\begin{proof}
If the involution is central on $M$  then $\mu(a) \in \Zent(M,*) \cap L = \Zent(L,*)$ for any $a \in L$, so the involution is super-central on $L$. Furthermore $\mu(aj) = \epsilon$ for $a \in L$, so~$\epsilon$ is symmetric, hence $\gamma$ is symmetric.
\end{proof}

\subsection{Antisymmetry and anticommutativity}\label{ASAC}

The notions presented in this subsection, although somewhat peculiar, will be needed in \Sref{sec:12}.

We say that a loop with involution is {\bf{anti-symmetric}} if $a^* \neq a$ for every noncentral element $a$. The passage of this property to the double is trivial.
\begin{prop}\label{AS'}
For $M = \CD(L,*,\gamma,\epsilon)$, $M$ is anti-symmetric if and only if $L$ is anti-symmetric and $\epsilon \neq 1$.
\end{prop}

We say that a loop $L$ is {\bf{anti-commutative}} if $[a,b] \neq 1$ whenever $a,b,ab \not \in \Zent(L)$. For example, an abelian group is anti-commutative (we believe the terminology nevertheless makes sense).
\begin{rem}
When $*$ is the identity map on $L$, $M = \CD(L,*)$ is anti-commutative if and only if $M$ is an abelian group.
%
\end{rem}

\begin{prop}\label{AC'}
Suppose $*$ is nonidentity on $L$. Then the loop $M = \CD(L,*)$ is anti-commutative if and only if $L$ is anti-commutative and is anti-symmetric.
\end{prop}
\begin{proof}
Since $a(bj) = (ba)j$ and $(bj)a = (ba^*)j$, $a$ and $bj$ commute if and only if $a^* = a$. Recall from \Cref{ZentMnot1} that $\Zent(M) \sub L$.
The conditions on $M$ include that $[a,bj] \neq 1$ whenever $a,bj,a(bj) \not \in \Zent(M)$, equivalently $a \not \in \Zent(L,*)$; so every symmetric element of~$L$ must be central.
\end{proof}

\section{Central-by-abelian doubles}\label{sec:cba}

Let $M = \CD(L,*,\gamma)$ be the Cayley-Dickson double of a loop $L$.
We find conditions under which $M$ is central-by-abelian. It will be useful to explicitly compute the commutators and associators in $M$.

\begin{prop}\label{computecomm}
We have that $[M,M] \sub L$. Explicitly, the commutators satisfy:
\begin{eqnarray}
&&[a,bj] \cdot ab^* = a^*b^*;  \label{Comm02}\\
&&[aj,b] \cdot b^*a^* = b a^*;  \label{Comm01}\\
&& b^*a = a^*b \cdot [aj,bj]  \label{Comm03}
\end{eqnarray}
for every $a,b \in L$. In particular $[bj,a^*] = [a,bj]$.
\end{prop}
\begin{proof}
By the definitions \eq{commdef} and \eqs{Mul1}{Mul3}:
$(ba)j = a(bj) = (bj)a \cdot [a,bj] = (ba^*)j \cdot [a,bj] = (ba^*\cdot [a,bj]^*)j$ so that
$ba = ba^*\cdot [a,bj]^*$ and we apply the involution; \
$(ab^*)j = (aj)b = b(aj) \cdot [aj,b] = (ab)j \cdot [aj,b] = (ab \cdot [aj,b]^*)j$ so that
$(ab^*) = (ab \cdot [aj,b]^*)$ and again we apply the involution; \
and $\gamma b^*a = (aj)(bj) = (bj)(aj) \cdot [aj,bj] = \gamma a^*b \cdot [aj,bj]$ so that
 $b^*a  = a^*b \cdot [aj,bj]$.
\end{proof}

\begin{cor}
$[M,M]$ is the subloop of $L$ generated by $[L,L]$ and all the elements $r_{ab}^{-1}(a^*b)$ and $\ell_a^{-1}(a^*)$ for $a,b \in L$ (where $r_a, \ell_a \co L \ra L$ are right and left multiplication by $a$).
\end{cor}

This has a useful consequence.
\begin{cor}\label{iff8}
We have that
\begin{enumerate}
\item $[M,M] \sub \Zent(L)$ if and only if $*$ is central on $L$;
\item\label{iff8.2} $[M,M] \sub \Zent(L,*)$ if and only if $*$ is super-central on $L$.
\end{enumerate}
\end{cor}
\begin{proof}
By \eq{Comm02} we have that $a^* = [a,j]a$, so if $[a,j] \in \Zent(L)$ the involution is central, and if $[a,j] \in \Zent(L,*)$ the involution is super-central. On the other hand if the involution is central (respectively, super-central) then, by \eqs{Comm02}{Comm03}, $[a,bj] = [bj,a]^{-1} = \mu(a)$ and $[aj,bj] = \mu(a)^{-2}\mu(ab)$ are central (respectively, central and symmetric) in $L$.
\end{proof}


Next, we want to express the associators in $M$ in terms of the commutators and associators in $L$.
By computations analogous to \Pref{computecomm} (or noting that a normal subloop $N \normali M$ satisfies $[M,M,M] \sub N$ if and only if $M/N$ is a group), we always have:
\begin{prop}\label{MMML}
$[M,M,M] \sub L$.
\end{prop}
For the associators of $M$ to be central in $L$ requires a somewhat stronger assumption on~$L$, as we show below.
\begin{prop}\label{Icomputeassoc}
Assume $L$ is central-by-abelian. Then the associators in $M = \CD(L,*)$ are as follows:
\begin{eqnarray*}
{}[a,b,c] & = &  [a,b,c] \\
{}[a,b,cj] & = &  [a^*,b^*,c^*][b^*,a^*] \\
{}[a,bj,c] & = & [c,a^*,b^*]^{-1} [a^*,c,b^*] [c,a^*] \\
{}[a,bj,cj] & = & [c^*,a,b]^{-1}[a,c^*,b][c^*,a][b,a] \\
{}[aj,b,c] & =  & [c,b,a^*]^{-1}[c,b] \\
{}[aj,b,cj] & = &  [c^*,b^*,a]^{-1}[a,b^*][c^*,b^*] \\
{}[aj,bj,c] & = & [b^*,a,c][b^*,c,a]^{-1}[b^*,c][a,c]\\
{}[aj,bj,cj] &  = & \gamma^* \gamma^{-1} [c^*,b,a^*]^{-1}\, [a^*b,c^*][a^*,b].
\end{eqnarray*}
\end{prop}
\begin{proof}
For each $x,y,z$ in $L \cup Lj$, Table~\ref{TAB} gives $(xy)z$ and $x(yz)$. When these two products are in $L$, we note that both are words on the same three letters from $\set{a,a^*,b,b^*,c,c^*}$, as a letter is starred in a product precisely when~$j$ appears to its left an odd number of times. We then find the associator using  \Fref{the12}. On the other hand when $(xy)z, x(yz) \in Lj$, write $(xy)z = uj$ and $x(yz) = u'j$ for $u,u' \in L$; by definition $uj = u'j \cdot [x,y,z] = (u'[x,y,z]^*)j$, so $u = u'[x,y,z]^*$ and we find $[x,y,z]^*$ in the same manner. We then apply \eq{commutator*} and \eq{associator*} to remove the external conjugates.
\end{proof}

\begin{prop}\label{Icomputeassoc-cor}
We have that $[M,M,M] \sub \Zent(L)$ if and only if $L$ is central-by-abelian.
\end{prop}
\begin{proof}
Let $a,b,c \in L$. By definition
\begin{eqnarray*}
(c\cdot ba[a,b])j & =&  (c(ab))j = (ab)(cj) = a(b(cj)) \cdot [a,b,cj] \\
& = & ((cb)a)j \cdot [a,b,cj] = (((cb)a)[a,b,cj]^*)j \\
& = & (((c(ba))[c,b,a])[a,b,cj]^*)j,
\end{eqnarray*}
so $c\cdot (ba)[a,b]  = ((c(ba))[c,b,a])[a,b,cj]^*$. Denoting $\pi = [c,b,a]$ and $\pi' = [a,b,cj]^*$, we have that $c\cdot (ba)[a,b]  = ((c\cdot (ba))\pi)\pi'$. Now, if $[L,L,L] \sub [M,M,M] \sub \Zent(L)$, then $\pi,\pi'$ are central and by canceling $c$, and then $ba$, we obtain that $[a,b] = \pi\pi'$ is central as well.

On the other hand, if $L$ is central-by-abelian then $[M,M,M]$ is contained in the associator-commutator subgroup of $L$ by the computation given in \Pref{Icomputeassoc}.
\end{proof}

\begin{cor}
Let $M = \CD(L,*)$.
We have that
\begin{equation}\label{Mprop}
[M,M],\ [M,M,M] \sub \Zent(L)
\end{equation}
if and only if $L$ is central-by-abelian and $*$ is central on $L$.
\end{cor}
\begin{proof}
First assume $[M,M],[M,M,M] \sub \Zent(L)$. Then $[L,L],[L,L,L] \sub \Zent(L)$ so $L$ is central-by-abelian, and $*$ is central by \Cref{iff8}. On the other hand, if $L$ is central-by-abelian then $[M,M,M] \sub \Zent(L)$ by \Pref{Icomputeassoc} and $[M,M] \sub \Zent(L)$ again by \Cref{iff8}.
\end{proof}

However, \eq{Mprop} does not guarantee that $M$ is central-by-abelian, which is equivalent to $[M,M], [M,M,M] \sub \Zent(M)$. Since $\Zent(M) \cap L = \Zent(L,*)$, this requires a yet stronger statement. Notice that $[M,M,M] \sub \Zent(M,*)$ is equivalent to $[M,M,M] \sub \Zent(M)$ by \Pref{MMML} and \eq{thesymmZent}.

\begin{thm}\label{main1}
Let $L$ be a central-by-abelian loop with an involution $*$, and let $M = \CD(L,*,\gamma)$ be the Cayley-Dickson double of~$L$, with the involution extended by \eq{extinv}.
Then:

$$\begin{array}{lcl}
\ARI{$[M,M] \sub \Zent(M)$} & \Longleftrightarrow & \ARI{$*$ is super-central on $L$};
\\
\ARI{$[M,M,M] \sub \Zent(M)$} & \Longleftrightarrow &
\ARII{$[L,L] \sub \Zent(L,*)$,}{$[L,L,L] \sub \Zent(L,*)$};
\\
\ARI{$M$ is central-by-abelian} & \Longleftrightarrow &
\ARIII{$L$ is central-by-abelian,}{$[L,L,L] \sub \Zent(L,*)$, and}{$*$ is super-central on $L$};
\\
\ARII{$M$ is central-by-abelian}{and $*$ is super-central on $M$}
& \Longleftrightarrow &
\ARIV{$L$ is central-by-abelian,}{$[L,L,L] \sub \Zent(L,*)$,}{$*$ is super-central on $L$,}{and $\gamma^* = \gamma$}.
\end{array}$$
\end{thm}
\begin{proof}
We know that $[M,M] \sub L$, so $[M,M] \sub \Zent(M)$ if and only if $[M,M] \sub \Zent(M) \cap L$, and likewise for $[M,M,M]$.
\begin{enumerate}
\item We have that $\Zent(M) \cap L = \Zent(L,*)$ by \eq{theZent}, so the claim follows by \Cref{iff8}\eq{iff8.2}.
\item Assume $[M,M,M] \sub \Zent(M)$. By \Pref{MMML}, $[M,M,M] \sub \Zent(M) \cap L = \Zent(L,*)$, and $L$ is central-by-abelian by \Pref{Icomputeassoc-cor}. Now $[L,L,L] \sub \Zent(L,*)$ follows from $[L,L,L] \sub [M,M,M]$, and the computation $[a^*,b^*,c^*j] = [a,b,c][b,a]$ given in the second line of the table in \Pref{Icomputeassoc} gives $[L,L] \sub \Zent(L,*)$.
    On the other hand if $L$ is central-by-abelian with symmetric commutators and associators, then the associators of $M$ are symmetric by inspection in the table given in \Pref{Icomputeassoc}; in the eighth row of that table we also use that $\gamma^*\gamma^{-1}$ is symmetric, see \Rref{ongamma}\eq{ongamma2}.
\item\label{ItemIII} Here we combine the previous two statements, noting that if $*$ is super-central then $[L,L] \sub \Zent(L,*)$ by \Pref{sc}.
\item Assume the conditions in \eq{ItemIII} hold. Then the function $\mu$ extends from $L$ to $M$ by $\mu(aj) = \epsilon$ for any $a \in L$. Thus $*$ is super-central on $M$ if and only if $\epsilon^* = \epsilon$, equivalently $\gamma^* = \gamma$ by \Pref{basiccond}.
\end{enumerate}
\end{proof}

By the remark after \Pref{thecenter}, if $M$ has central associators then the associators are symmetric, so in the language of varieties, the final part of \Tref{main1} reads as follows:
\begin{cor}
The variety of central-by-abelian loops with symmetric associators and super-central involutions is closed under Cayley-Dickson doubling with symmetric $\gamma$.
\end{cor}

With this, we can construct a successive chain of Cayley-Dickson doubles which are all central-by-abelian, thus amenable to explicit computations.

\section{Properties of the double and the derivative}\label{PII}

For algebras, the Cayley-Dickson doubling process constructs from the real numbers the well-known chain $\R \sub \C \sub \HQ \sub \O$. The double of an algebra $A$ is commutative if and only if the involution on $A$ is trivial; the double is associative if and only if $A$ is commutative; and the double is alternative if and only if $A$ is associative. After developing the notions of derivative in this section, and homogeneity in the next section, in \Sref{sec:double} we record similar statements for loops.

By Birkhoff's theorem (\eg\ \cite{Berg}), a variety of loops with involution is defined by identities. Notice that (with the exception of anti-symmetric and anti-commutative loops), all types of loops or involutions discussed so far are indeed defined by identities, taking into account that the language of loops includes the binary operations $\ell_a^{-1}(b) = a\! \setminus\! b$ and $r_{a}^{-1}(b) = b\, /\, a$ (for all the varieties in this paper, the unary left- and right- inverse operations suffice).
\begin{rem}\label{justnote}
Throughout this paper, the term ``variety'' always refers to a variety of loops with involution. This is likewise true when the variety is defined in terms of the loop without any explicit mention of the involution.
\end{rem}

\begin{defn}
Let $\V$ be a variety of loops with involutions. We define the {\bf{derivative}} $\V'$ to be the variety of all loops $(L,*)$ such that $\CD(L,*,1,1) \in \V$ (this is the double with $\gamma = \epsilon = 1$).
\end{defn}

Since $\CD(L)/L \isom \Z_2$, the idea of derivatives is only meaningful for varieties containing~$\Z_2$:
\begin{prop}
If $\Z_2 \in \V$ then $\Z_2 \in \V'$ as well. On the other hand if $\Z_2 \not \in \V$ then $\V' = \emptyset$.
\end{prop}
\begin{proof}
A variety which contains any Cayley-Dickson double of any loop must contain the quotient $\CD(L)/L \isom \Z_2$. Therefore, if $\Z_2 \not \in \V$ then $\V' = \emptyset$. On the other hand if $\Z_2 \in \V$ then $\CD(\Z_2,*,1,1) = \Z_2 \times \Z_2 \in \V$  by \Eref{trivi},
and $\Z_2 \in \V'$ as well.
\end{proof}

\begin{rem}
The variety (of loops with involution) generated by $\Z_2$ is the variety of
abelian groups of exponent $2$ with trivial involution.
\end{rem}

The variety $\IDEN$ of (commutative) loops with the identity involution, also plays a special role:
\begin{prop}\label{subI}
Suppose $\Z_2 \in \V \sub \IDEN$. Then $\V' = \V$.
\end{prop}
\begin{proof}
Let $L \in \V$. Then the identity on $L$ is trivial so $\CD(L,*,1) = L \times \Z_2$ by \Eref{trivi}, and $\CD(L,*,1) \in \V$, proving that $L \in \V'$.
\end{proof}

Our focus is, therefore, on varieties which contain $\Z_2$ but are not contained in $\IDEN$.

\section{Homogeneous varieties}\label{badsec}

In the context of the derivative, it is convenient to know whether or not the condition $\CD(L,*,\gamma,\epsilon) \in \V$ depends on the choice of $\gamma,\epsilon$. This section provides the technical tool for answering this question.

A {\bf{term}} in the language of quasigroups with involution is either a letter, or one of the following
$$\phi_1 \phi_2, \qquad \phi_1 / \phi_2, \qquad \phi_1 \bs \phi_2, \qquad \phi^*$$
where $\phi, \phi_1,\phi_2$ are terms.
(Recall that $(x/y)y = x = y (y \bs x)$, so that $x^\lam = 1/x$ and $x^{\rho} = x \bs 1$).

\begin{prop}
Any term in the language of quasigroups with involution over the letter set $X$ is equal to a term in the language of quasigroups over the letter set $X \cup X^*$ (where $X^* = \set{x^* \suchthat x \in X}$).
\end{prop}
\begin{proof}
We have that
$y^*(x/y)^* = x^* = (y \bs x)^*y^*$,
hence
$$(x/y)^* = y^* \bs x^*, \qquad (y \bs x)^* = x^* / y^*,$$
so the claim follows by induction on the structure.
\end{proof}

Let us record some basic computations in the double of a loop $L$:
\begin{eqnarray}
&a / bj = (\gamma b^* \bs a)j, \qquad
&bj / a = (b / a^*)j, \qquad\qquad
aj / bj = b \bs a,\\
&a \bs bj = (b/a)j, \qquad
&bj \bs a = (\gamma^*b^* \bs a^*)j, \qquad
aj \bs bj = (b^*/a^*). 
\end{eqnarray}

\begin{defn}
Let $T$ denote the collection of terms in the language of quasigroups with involution over a countable set of letters $a_\alpha,\dots$ and $j$.

We define a degree function $\deg\!\co T \ra \set{0,1} \times \Z \times \Z_2$ as follows. For a letter $a$,
$$\deg_j(a) = 0, \qquad \deg_j(j) = 1;$$
$$\deg_\epsilon(a) = 0, \qquad \deg_\epsilon(j) = 0;$$
$$\deg_\gamma(a) = 0, \qquad \deg_\gamma(j) = 0.$$
Then, by induction on the structure,
$$\deg_j(\phi_1\phi_2) = \deg_j(\phi_1)+\deg_j(\phi_2),$$
$$\deg_j(\phi_1 / \phi_2) = \deg_j(\phi_1) - \deg_j(\phi_2),$$
$$\deg_j(\phi_2 \bs \phi_1) = \deg_j(\phi_1) - \deg_j(\phi_2),$$
$$\deg_j(\phi^*) = \deg_j(\phi);$$
$$\deg_\epsilon(\phi_1\phi_2) = \deg_\epsilon(\phi_1)+\deg_\epsilon(\phi_2),$$
$$\deg_\epsilon(\phi_1 / \phi_2) = \deg_\epsilon(\phi_1) - \deg_{\epsilon}(\phi_2),$$
$$\deg_\epsilon(\phi_2 \bs \phi_1) = \deg_\epsilon(\phi_1) - \deg_{\epsilon}(\phi_2),$$
$$\deg_\epsilon(\phi^*) = \deg_j(\phi);$$ 
$$\deg_\gamma(\phi_1\phi_2) = \deg_\gamma(\phi_1)+\deg_\gamma(\phi_2) + \deg_j(\phi_1)\deg_j(\phi_2),$$
$$\deg_\gamma(\phi_1 / \phi_2) = \deg_\gamma(\phi_1) - \deg_\gamma(\phi_2) - \deg_j(\phi_2)(1-\deg_j(\phi_1)),$$
$$\deg_\gamma(\phi_2 \bs \phi_1) = \deg_\gamma(\phi_1) - \deg_\gamma(\phi_2) - \deg_j(\phi_2)(1-\deg_j(\phi_1)),$$
$$\deg_\gamma(\phi^*) = \deg_\gamma(\phi).$$
\end{defn}

Now let $\phi$ be any term, and $(L,*)$ any loop with involution.
Let $\pi \co \set{a_{\alpha}} \ra L$ be any substitution map. Let $\pi_{1,1}(\phi)$ be the evaluation of $\phi$ in the Cayley-Dickson double $\CD(L,*,1,1)$, where the substitution is extended to the letter set by $j \mapsto j$. Let $\pi_{\gamma,\epsilon}$ be the evaluation of $\phi$ in $\CD(L,*,\gamma,\epsilon)$, where again $j \mapsto j$.

Induction on the structure shows that
$$\pi_{\gamma,\epsilon}(\phi) = \pi_{1,1}(\phi) \gamma^{\deg_{\gamma}(\phi)} \epsilon^{\deg_{\epsilon}(\phi)},$$
which leads to the following corollary.
\begin{defn}
An identity $\phi_1 = \phi_2$ is {\bf{homogeneous}} if $\deg_\gamma(\phi_1) = \deg_\gamma(\phi_2)$ and $\deg_\epsilon(\phi_1) = \deg_\epsilon(\phi_2)$.
\end{defn}

\begin{cor}
The validity of a homogeneous identity $\phi_1 = \phi_2$ in $\CD(L,*,\gamma,\epsilon)$ does not depend on $\gamma$ or $\epsilon$.
\end{cor}

We say that a variety $\V$ is {\bf{homogeneous}} if it is defined by homogeneous identities. We finally arrive at the reason for this lengthy definition.
\begin{cor}\label{homogencor}
Let $\V$ be a homogeneous variety. Then, for any loop $L$, $\CD(L,*,1,1) \in \V$ implies that $\CD(L,*,\gamma,\epsilon) \in \V$ for every symmetric $\gamma$ and~$\epsilon$.
\end{cor}

An easy criterion for homogeneity is that every identity which holds in abelian groups with a generic involution is homogeneous. Consequently, we observe that all the identities and properties discussed in this paper are homogeneous in this sense, with a unique exception: the variety $\IDEN$ defined by the identity $x^* = x$, which is not homogeneous in terms of~$\epsilon$, and indeed the involution on the double
$\CD(L,\id,\gamma,\epsilon)$ is not trivial once $\epsilon \neq 1$.

A useful summary is that if $\V$ is a homogeneous variety, then $L \in \V'$ implies that $\CD(L,*,\gamma,\epsilon) \in \V$ for any symmetric $\gamma,\epsilon$. The converse also holds, and may be used to construct counterexamples.

\begin{rem}
A stronger sense of homogeneity can be defined by the condition that if $L \in \V'$ then $\CD(L,*,\gamma,\epsilon) \in \V$ for any $\gamma,\epsilon$ satisfying \eq{thecond}, without the restriction that $\gamma$ and~$\epsilon$ are symmetric. However, varieties satisfying this stronger condition are quite rare. Moreover, note that by \Cref{Inv-perfect}, if $\CD(L,*,\gamma)$ has well-defined inverses, then $\gamma^* = \gamma$. This is restrictive in the first doubling. On the other hand, by \Cref{ZentMnot1} all elements in the center of $\CD(L,*,\gamma)$ are symmetric (unless $*$ is the identity on $L$), so when we consider the double of the double and so on, the restriction on $\gamma$ being symmetric becomes essentially redundant.
\end{rem}


\section{Properties of the double}\label{sec:double}

Let us now present some examples.
Since a loop $L$ is contained in the double $\CD(L)$, we clearly have that $\V' \sub \V$; namely, $\V'$ is defined by additional identities. In principle it is not hard to compute~$\V'$: substitute elements of $L$ and $Lj$ for any variable in the defining identities of $\V$ (see
\cite[Section~7]{KPVonallCD} for a detailed description). Derivation is monotone with respect to inclusion, and respects intersections.

\medskip
Of particular interest are varieties $\V$ for which $\V' = \V$, which we call {\bf{perfect}}. For then, if $\V$ is homogeneous, we can iterate the construction with any series of symmetric $\gamma_i$, and never leave~$\V$. Another application of the derivative is that when homogeneous varieties $\V_1 \sub \V_2$ have $\V_1' = \V_2'$, it follows that any Cayley-Dickson double in $\V_2$ is in fact in~$\V_1$.

\medskip
\def\derrow{\xymatrix{{} \ar@{->}[r]|\circ & {}}}
We use the notation $\V \derrow \V'$ for the passage from $\V$ to its derivative, which is always a subvariety. When a chain of derivations terminates, we write $\V \odot$ to denote that $\V$ is perfect and homogeneous, and $\V \otimes$ to denote that $\V$ is perfect but not homogeneous.

\begin{prop}
Let $\V$ be a finitely based variety containing $\Z_2$. Then the chain
$$\V \derrow \V' \derrow \V'' \derrow \cdots$$
terminates in a nonempty variety, which is the unique maximal perfect subvariety of~$\V$.
\end{prop}
\begin{proof}
If $\V$ is defined by a family of identities $\set{\phi_{\lam}}_{\lam \in \Lambda}$, then $\V'$ is defined
by the identities obtained from each $\phi_\lam$ by substituting $a_t$ or $a_tj$ for each variable $x_t$, and interpreting the identity in $\CD(L,*,1,1)$. Every $\phi_\lam$ is thus replaced by $2^k$ identities of the same length, where $k$ is the number of variables in $\phi_{\lam}$. Since the number of identities of a given length is bounded, this process must terminate once we assume the length of identities in $\set{\phi_{\lam}}$ is bounded.

For the final statement, note that if $\W \sub \V$ is perfect then $\W = \W' \sub \V'$.
\end{proof}

If $\phi_1,\dots,\phi_t$ are identities, we write $\set{\phi_1,\dots,\phi_t}$ to denote the variety defined by these identities, and use common names of properties when there is no room for misinterpretation.

\begin{exmpl}\label{ExZA}
Denote by $\var{ZG}$ and $\var{ZC}$ the varieties of loops (with involution) which are central-by-groups and central-by-(commutative loops), respectively; and let $\var{ZA} = \var{ZG} \cap \var{ZC}$ be the variety of central-by-abelian loops. In terms of identities, $\var{ZG}$ is defined by $[L,L,L] \sub \Zent(L)$, and $\var{ZC}$ by $[L,L] \sub \Zent(L)$.

Let $\var{ZG}_0$ and $\var{ZC}_0$ denote the varieties of loops with involution satisfying $[L,L,L] \sub \Zent(L,*)$ and respectively $[L,L] \sub \Zent(L,*)$, and $\var{ZA}_0 = \var{ZG}_0 \cap \var{ZC}_0$ their intersection. Let $\Central$ and $\superCentral$ be the varieties of loops with central and super-central involutions, respectively. The varieties appear in the diagram below, with dashed lines going down from a variety to a subvariety.
$$
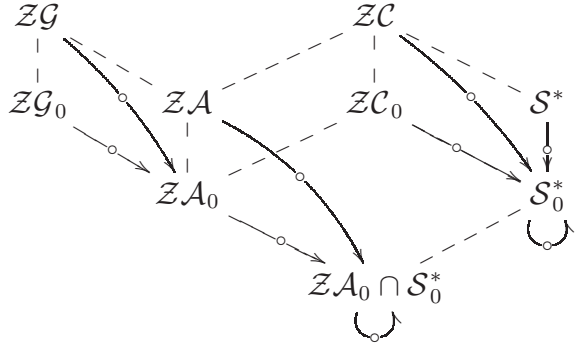
\xymatrix@C=26pt@R=16pt{
\var{ZG} \ar@/^1ex/@{->}[rdd]|{\circ} \ar@{--}[rd] \ar@{--}[d] & {} & \var{ZC} \ar@{--}[d] \ar@/^1ex/@{->}[rdd]|{\circ} \ar@{--}[ld] \ar@{--}[rd] & {} \\
\var{ZG}_0 \ar@{->}[rd]|{\circ} & \var{ZA}\ar@{--}[d] \ar@/^2ex/@{->}[rdd]|{\circ} & \var{ZC}_0 \ar@{->}[rd]|(0.475){\circ} \ar@{--}[rd] \ar@{--}[ld] & \Central \ar@{->}[d]|{\circ} \\
{} & \var{ZA}_0  \ar@{->}[rd]|{\circ} & {} & \superCentral \ar@{--}[ld] \perferrow   \\
{}&{}& \var{ZA}_0 \cap \superCentral \perferrow
&{} \\
{} & {} & {} & {}
}$$
In \Sref{sec:cba} we proved that the derivatives of these varieties are as depicted in the diagram. All these varieties are homogeneous.
\end{exmpl}

\begin{rem}\label{N'}
We saw in \Rref{N'0} that the variety of normal involutions is perfect.
\end{rem}

\begin{exmpl}\label{der1}
We have that:
\begin{eqnarray*}
\set{\mbox{commutative}}  &\derrow & \set{\mbox{$*$ is the identity}} \otimes \\
\set{\mbox{associative}} & \derrow & \set{\mbox{associative and commutative}} \\
& \derrow & \set{\mbox{associative and $*$ is the identity}} \otimes
\end{eqnarray*}
(The first claim is \Pref{dercomm}, and the derivative of associativity can be read from \Rref{nucMall}).
\end{exmpl}

A loop $L$ is {\bf{left alternative}} if $[x,x,y] = 1$ for all $x,y \in L$, {\bf{right alternative}} if $[x,y,y] = 1$, and {\bf{alternative}} if both identities hold.
\begin{prop}\label{510}
For a loop with involution, being left- (or right-) alternative implies alternativity.
Let $\FLEX$ and $\ALT$ denote the varieties of flexible and alternative loops with involution, respectively. Also let
\begin{eqnarray*}
\var{P}_1 & = & \set{r_{aa^*} = \ell_a \ell_{a^*} = \ell_{a^*} \ell_a} =
\set{\ell_{aa^*} = r_a r_{a^*} = r_{a^*} r_a}; \\ 
\var{P}_2 & = & \set{\ell_{aa^*} = \ell_a \ell_{a^*} = \ell_{a^*} \ell_a}
= \set{r_{aa^*} = r_a r_{a^*} = r_{a^*} r_a}; \\ 
\var{Q} &= &\set{\ell_a\ell_{a^*}= \ell_{a^*}\ell_{a}=r_{a^*}r_a = r_a r_{a^*}} 
\end{eqnarray*}
Then
\begin{enumerate}
\item The variety $\set{(aa)a=a(aa)}$ is perfect.
\item $\var{Q}$ is perfect.
\item $\var{P}_1 \derrow \var{P}_1 \cap \var{P}_2$\ \  and\ \ $\var{P}_2 \derrow \var{P}_1 \cap \var{P}_2$.
\item $\var{P}_1 \cap \var{Q} = \var{P}_2 \cap \var{Q} = \var{P}_1 \cap \var{P}_2$ is perfect.
\item $\FLEX \derrow  \FLEX \cap \var{Q} \, \odot$. %
\item\label{6here} $\ALT \derrow \ALT \cap \var{P}_1 \derrow (\ALT \cap \var{P}_1 \cap \var{P}_2) \odot$.
\end{enumerate}
\end{prop}
\begin{proof}
We verify the first step of \eq{6here}, leaving the other verifications to the interested reader. We compute $\ALT'$. Let $L$ be any loop with involution. For $\CD(L,*,1)$ to satisfy the left alternative identity $x(xy)=(xx)y$ (for all $x,y \in \CD(L,*,1)$), $L$ must satisfy the four conditions
$a(ab)=(aa)b$, $a(a(bj))=(aa)(bj)$, $(aj)((aj)b)=((aj)(aj))b$, and $(aj)((aj)(bj))=((aj)(aj))(bj)$ (for all $a,b \in L$). These are equivalent to
$a(ab)=(aa)b$, $(ba)a=b(aa)$, $(ba^*)a=(a^*a)b$, and $a(a^*b)=b(a^*a)$, respectively. The former two are alternativity, and the latter two are the right-hand side definition of $\var{P}_1$, which is equivalent to the left-hand side version by applying the involution.
\end{proof}


Next we consider some weak versions of the inverse property, see~\cite{4halves}.
\begin{exmpl}
Consider the following varieties of loops with involution:
\begin{enumerate}
\item $\var{H}$ is the variety of loops in which $(xy)^{\lam^2} = x^{\lam^2}y^{\lam^2}$ where $\lam$ is the left inverse; equivalently for the right inverse~$\rho$.
\item $\INV$ is variety of loops in which the inverse is defined (\ie\ $x^{\lam} = x^{\rho}$).
\item $\var{WI}$ (weak inverse property) is the varity of loops in which  $(xy)z=1$ if and only if $x(yz)=1$ (this can be defined by an identity, \eg\ $x(y(xy)^{\rho})=1$).
\item $\var{IP}$ (inverse property) is the varieties of loops satisfying $x^{\lam}(xy)=y$ and $(xy)y^{\rho} = x$ (which are equivalent for loops with involution).
\item $\var{AA}$ (antiautomorphic inverse) is defined by $(xy)^{\lam} = y^{\lam}x^{\lam}$.

\end{enumerate}
Then:
\begin{enumerate}
\item $\var{H} \derrow \INV \odot$ (see \Cref{Inv-perfect}).
\item $\var{WI} \derrow \var{WI} \cap \INV \odot$.
\item $\var{AA}$ is perfect.
\item $\var{IP}$ is perfect.
\end{enumerate}
\end{exmpl}

\section{Moufang doubles}\label{PIII}

The goal of this section is to assert the conditions on $(L,*)$ under which $\CD(L,*)$ is Moufang, thereby generalizing \Rref{onChein}.

Recall that a loop is Moufang if it satisfies the equivalent identities
\begin{eqnarray}
z(x(zy))  &= &((zx)z)y, \label{Mouf1}\\
((xz)y)z & = & x(z(yz)) \label{Mouf2}\\ 
(zx)(yz) & = & (z(xy))z.\label{Mouf3}
\end{eqnarray}
A Moufang loop is diassociative, and in particular alternative and flexible. By Moufang's theorem, if $[a,b,c]=1$ in a Moufang loop, then $\sg{a,b,c}$ is a group.

\begin{thm}\label{MoufCD}
Let $(L,*)$ be a loop with involution, and $\gamma \in \Zent(L,*)$. Then $\CD(L,*,\gamma)$ is Moufang if and only if
\begin{enumerate}
\item $L$ is Moufang,
\item $[a,cc^*]=1$ for every $a,c \in L$,
\item $[c,c^*]=1$ for every $c$,
\item $[a,c,c^*] = 1$ for every $a,c$,
\item every $cc^*c$ is in the nucleus.
\end{enumerate}
\end{thm}
\begin{proof}
For $\CD(L,*,\gamma)$ to be Moufang it is necessary and sufficient that \eq{Mouf3} holds for the $2^3$ choices of $x,y,z$ from $L \cup Lj$. The resulting conditions are that~$L$ is Moufang, and satisfies the identities
\begin{eqnarray}
(ac)(c^*b) & = & ((ab)c^*)c \label{MoufL1}\\ 
(ac)(c^*b) & = & ((ab)c)c^* \label{thisoneplease} \\ 
(ac)(c^*b) & = & c(c^*(ab)) \label{MoufL3} \\ 
(ac)(c^*b) & = & c^*(c(ab)). \label{MoufL4} 
\end{eqnarray}

First assume that $\CD(L,*,\gamma)$ is Moufang. As a subloop, $L$ is Moufang as well.
Taking $b = 1$ in \eqs{MoufL1}{MoufL4}, \eq{thisoneplease} becomes redundant, and the other identities imply $r_{c^*}r_c = r_cr_{c^*} = \ell_c\ell_{c^*} = \ell_{c^*}\ell_c$. In particular, $cc^* = c^*c$. Since $L$ is flexible, we also have that $[\ell_c,r_c] = 1$.
Now $1 = [\ell_c,\ell_c\ell_{c^*}] = [\ell_c,r_cr_{c^*}] = [\ell_c,r_{c^*}]$. To summarize, we proved that
$$\ell_c, r_c, \ell_{c^*}, r_{c^*} \qquad \mbox{commute}.$$
Since $[\ell_c,r_{c^*}] = 1$ we have that $[c,x,c^*] = 1$. Thus, by Moufang's theorem, $\sg{c,c^*,x}$ is associative, and in particular $[x,c^*,c]=[c,c^*,x] = 1$. Therefore $(cc^*)x = c(c^*x) = \ell_c\ell_{c^*}(x)=r_{c^*}r_{c}(x)=(xc)c^*=x(cc^*)$,
proving that 
$[cc^*,x] = 1$ for every $c \in L$. Now \eqs{MoufL1}{MoufL4} transform to
\begin{equation}\label{MoufLnew}
(ac)(c^*b) = (ab)cc^*.
\end{equation}
We now adjust Chein's proof \cite[Theorem~1]{Chein} (where the involution is assumed to be $x^* = x^{-1}$).
Multiply \eq{MoufLnew} by $c$ from the right to get
$$a(b(cc^*c)) = a(c(c^*b)c) \stackrel{\eq{Mouf2}}{=} ((ac)(c^*b))c \stackrel{\eq{MoufLnew}}{=} ((ab)cc^*)c \stackrel{\eq{Mouf2}}{=} (ab)cc^*c,$$
proving $[a,b,cc^*c] = 1$. Since the one-sided nuclei coincide, every $cc^*c \in \Nuc{L}$, so we proved all the conditions (1)--(5).

In the other direction, assume $L$ satisfies the conditions (1)--(5) of the theorem. Since $\sg{a,c,c^*}$ is associative, each of the identities \eqs{MoufL1}{MoufL4} coincides with \eq{MoufLnew}. This identity follows from $(a(b(cc^*c)) = (ab)(cc^*c)$ by reversing the computation, applying \eq{Mouf2} which is known. Therefore \eq{Mouf3} hold for all $x,y,z \in \CD(L,*,\gamma)$.
\end{proof}

\begin{rem}\label{Shtut}
Let $*$ be a normal involution on $L$, and assume \Eq{MoufLnew} holds. By taking $b = a^*$, \Eq{MoufLnew} implies that $(ac)(ac)^* = (ac)(c^*a^*) = (aa^*)(cc^*)$, so the function $\nu \co c \mapsto cc^*$ is a norm homomorphism to the center, and $L$ has the anti-automorphic inverse  property by \Pref{normalhom}.
\end{rem}

\begin{rem}
After computing the derivative of the variety $\MV$ of Moufang loops with involution
to be
$$\MV' = \set{\mbox{Moufang, $[c,c^*]=[a,cc^*]=[a,c,c^*]=[a,b,cc^*c]=1$}},$$
it would be curious to find the second derivative. Imposing $c(c^*c)$ to be in the nucleus for $\CD(L,*,\gamma)$ implies that $L$ is commutative by \Rref{nucMall}, since $(cj)(cj)^*(cj) \in Lj$. Indeed, commutativity of $L$ follows from \eq{MoufLnew} applied to the double, if we take $c=j$.
Consequently,
$$\MV'' = \set{\mbox{commutative Moufang, $[a,c,c^*]=[a,b,cc^*c]=1$}}.$$
Next, using \Eref{der1}, $\MV''' = \MV \cap \IDEN$ is the variety of commutative Moufang loops with the identity involution, which is perfect (but not homogeneous, see \Cref{homogencor}).
\end{rem}
\begin{proof}
We supply some more details for the computation of the derivatives. The varieties $\set{[c,c^*]=1}$ and $\set{[a,cc^*]=1}$ are perfect. The derivative of $\set{[c,a,c^*]=1}$ (which has the same intersection with $\MV$ as $\set{[a,c,c^*]=1}$) is $\set{[c,a,c^*]=1,\, a(a^*c)=(ca)a^*}$, which is perfect. The derivative of the variety $\set{[a,b,c(c^*c)]=1}$ is $\set{[a,b]=[a,b,c(c^*c)]=1}$.

In the passage from $\MV''$ to $\MV'''$ we use the fact that in a commutative Moufang loop, every $c^3$ is in the nucleus \cite[Lemma~VII.5.7]{Bruck}.
\end{proof}

Our notation for iterative Cayley-Dickson doubling should be self-explanatory.
\begin{cor}\label{triple}
Let $(L,*)$ be a loop with involution. The following are equivalent:
\begin{enumerate}
\item $\CD^3(L,*;\gamma_1,\gamma_2,\gamma_3)$ is Moufang for some $\gamma_1,\gamma_2,\gamma_3 \in \Zent(L,*)$;
\item $\CD^3(L,*;\gamma_1,\gamma_2,\gamma_3)$ is Moufang for every $\gamma_1,\gamma_2,\gamma_3 \in \Zent(L,*)$
\item $L$ is a commutative Moufang loop and the involution is the identity on $L$.
\end{enumerate}
\end{cor}

Specializing in \Tref{MoufCD}, we have:
\begin{cor}
Let $L$ be a loop with a central involution, and $\gamma \in \Zent(L,*)$.
\begin{enumerate}
\item Then $\CD(L,*,\gamma)$ is Moufang if and only if $L$ is Moufang, every~$c^2$ is in the commutative center, and every $c^3$ is in the nucleus.
\item Assume $L$ is commutative. Then $\CD(L,*,\gamma)$ is Moufang if and only if $L$ is Moufang.
\end{enumerate}
\end{cor}

\begin{cor}\label{fin1}
Assume $L$ is associative and $\gamma \in \Zent(L,*)$. Then $\CD(L,*,\gamma)$ is Moufang if and only if the involution is normal.
\end{cor}
\begin{proof}
When $L$ is associative, the associator conditions in \Tref{MoufCD} are redundant. The condition $[a,cc^*]=1$ implies the involution is normal, and $[c,c^*]=1$ holds by \Rref{normalInv}.
\end{proof}

In another vein, loops with normal involution suggest a generalization for Chein's theorem~\cite{Chein}.
\begin{cor}\label{fin2}
Let $L$ be a loop with a normal involution, and $\gamma \in \Zent(L,*)$.
Then $\CD(L,*,\gamma)$ is Moufang if and only if $L$ is associative.
\end{cor}
\begin{proof}
Condition (5) in \Tref{MoufCD} requires that $cc^*c$ is in the nucleus, but $c^*c \in \Zent(L)$ when the involution is normal. Then, since $L$ is associative, the one-sided inverses coincide so $[c,c^*]=1$ by \Rref{normalInv}.
\end{proof}

Another way to read Corollaries~\ref{fin1}~and~\ref{fin2} is that
\begin{equation}\label{corcor}
\MV' \cap \ASS = \MV' \cap \Normal,
\end{equation}
where $\ASS$ is the variety of associative loops with involution, and $\Normal$ is the variety of loops with normal involutions. Furthermore, let $\COMM$ be the variety of commutative loops with involution.
In \Eref{der1} we pointed out that $\var{A}' = \var{A} \cap \var{C}$ and $\var{C}' = \IDEN$. The chain $\ASS \derrow \ASS \cap \COMM \derrow \ASS \cap \IDEN$ gives information only on loops obtained by doubling twice a commutative loop with the identity involution, when a description of three consecutive doublings, similar to \Cref{triple}, would be desirable.

Indeed, since $\Normal' = \Normal$ as stated in \Rref{N'}, we have the chain
$$\MV \cap \Normal \derrow \MV' \cap \Normal \derrow \MV'' \cap \Normal  \derrow \MV''' \cap \Normal,$$
whose members have convenient characterizations following from \eq{corcor}, namely $\MV'' \cap \Normal = \MV'' \cap \ASS = \ASS \cap \COMM$ and $\MV''' \cap \Normal = \MV''' \cap \ASS = \ASS \cap \IDEN$. As an addendum to \Cref{triple}, we now proved:
\begin{cor}
$\CD^3(L,*;\gamma_1,\gamma_2,\gamma_3)$ is Moufang with a normal involution 
if and only if $L$ is an abelian group with the identity involution. 
\end{cor}

\section{Loops which are central-by-(elementary abelian $2$-groups)}\label{sec:8}

If a loop $L$ is a central extension of a loop $A$, then a Cayley-Dickson double of $L$ is a central extension of $A \times \Z_2$. This motivates a closer look at loops which are central-by-(elementary abelian $2$-groups).
Recall that we denote by $\var{ZA}$ the variety of central-by-abelian loops. Let $\Exptwo$ denote the variety of loops (with involution) such that every $x^2$ is central. Therefore, our main concern from now on is the variety $\var{ZA} \cap \Exptwo$.

\begin{prop}
Every $L \in \var{ZA} \cap \Exptwo$ is power-associative.
\end{prop}
\begin{proof}
Indeed the inverse of $x$ is well defined as $(x^2)^{-1}x$, and every word in $x$ which is not $1$ or $x$ involves $x^2$ which is central, so can be reduced to a shorter word.
\end{proof}

\begin{rem}
Any anti-commutative loop (\Ssref{ASAC}) is in $\Exptwo$ (because $[a,a]=1$ implies $a\in \Zent(L)$ or $aa \in \Zent(L)$).
\end{rem}

\subsection{Central and normal involutions}

Recall that an involution on~$L$ is central if $x^* \in \Zent(L) x$, and normal if $x^*x \in \Zent(L)$. Also recall that the varieties of loops with central and normal involutions are denoted by $\Central$ and $\Normal$, respectively. These properties are closely related to $\Exptwo$:
\begin{prop}\label{coolrem0}
Any two of the following properties imply the third:
\begin{enumerate}
\item The involution is central ($\Central$).
\item The involution is normal ($\Normal$).
\item Every $x^2 \in \Zent(L)$ (this is $\Exptwo$).
\end{enumerate}
In other words, 
\begin{equation}\label{coolrem0succ}
\Central \cap \Normal = \Central \cap \Exptwo = \Normal \cap \Exptwo.
\end{equation}
\end{prop}
\begin{proof}
First assume the involution is central. Writing $x^* = \mu(x)x$ for some $\mu \co L \ra \Zent(L)$, the involution is normal if and only if $\mu(x)x^2 = x^*x \in \Zent(L)$; so $(1)+(2)=(1)+(3)$. Likewise, if the involution is normal and $x^2 \in \Zent(L)$, then $x^{-2}x = x^{-1} =  x^*(xx^*)^{-1}$ so
$x^* = x^{-2}(xx^*)x \in \Zent(L)x$ and $(2)+(3)\implies (1)$.
\end{proof}

In particular,
$$\var{ZA} \cap \Central \cap \Normal \sub \var{ZA} \cap \Exptwo.$$

\subsection{Properties of the double}

Let $M =\CD(L,*,\gamma)$. Since $(aj)^2 = a^*a$, $M \in \Exptwo$ if and only if $L \in \Exptwo \cap \Normal$. Namely,
\begin{equation}\label{E2'}
\Exptwo' = \Normal \cap \Exptwo.
\end{equation} \Rref{N'} then shows that this variety is perfect.

Since we are mostly concerned with the perfect variety $\var{ZA}_0 \cap \superCentral$ (see \Eref{ExZA}), we summarize this observation by noting that
\begin{equation}\label{cool}
\var{ZA}_0 \cap \superCentral \cap \Normal = \var{ZA}_0 \cap \superCentral \cap \Exptwo 
\end{equation}
is also perfect. 

\section{Diassociativity}\label{sec:dias}

We now add a new ingredient. Recall that a loop is {\bf{diassociative}} if every subloop generated by two elements is a group. Kowalski~\cite{Kow} proved that the variety of diassociative loops is not finitely based. He proves this relative to power associativity and existence of the inverse, so the same claim holds for the variety~$\DIAS$ of diassociative loops with involution.
Nevertheless, we have:
\begin{prop}\label{fb}
$\var{ZA} \cap \Exptwo \cap \DIAS$ is finitely based.
\end{prop}
\begin{proof}
A loop is diassociative if $\sg{x,y}$ is a group for any $x,y \in L$.
But if $L$ is central-by-abelian with $Z = \Zent(L)$, and $L/Z$ is of exponent $2$, then $\sg{x,y} \sub Z \cup Zx \cup Zy \cup Zxy$. Every associator in this latter set is of the form $[u,v,w]$ where $u,v,w \in \set{1,x,y,xy}$, so there are $4^3$ identities to verify.
\end{proof}

Assume $L \in \var{ZA}$. Dividing \eq{M2} by \eq{M1} and changing variables, we obtain
\begin{equation}\label{commlin}
[a,bc] = [a,b][a,c] \cdot [a,c,b]^{-1} [c,b,a]^{-1} [c,a,b].
\end{equation}
In particular, if $L \in \var{ZA} \cap \DIAS$, then
\begin{equation}\label{commlincor}
[a,ac] = [a,c].
\end{equation}

\begin{thm}
We have that $\var{ZA} \cap \Exptwo \cap \Central \cap \DIAS \sub \DIAS'$.
\end{thm}
\begin{proof}
Let $L \in \var{ZA} \cap \Exptwo \cap \Central \cap \DIAS$, and $M = \CD(L,*,\gamma)$ a Cayley-Dickson double, where $\gamma^* = \gamma$. By \Pref{coolrem0} and \eq{E2'}, $L \in \Normal \cap \Exptwo = \Exptwo'$, and so $M \in \Exptwo$.
As in \Pref{fb}, to verify that $M$ is diassociative it suffices to prove for any $x,y \in M$ that $[u,v,w] = 1$ for $u,v,w \in \set{1,x,y,xy}$. If $x,y \in L$, we know that the associators are trivial since $L \in \DIAS$.
Otherwise, by changing variables we may assume $x = x_0$ and $y = y_0j$ for $x_0,y_0 \in L$. Letting $Z = \Zent(M)$, we have that $\sg{x_0,y_0j} \sub Z \cup Z x_0 \cup Z y_0 j \cup Z (x_0 y_0) j$. Let $u,v,w \in \set{x_0,y_0j ,(x_0y_0)j}$. One computes all the associators by passing over the list of associators in $M$ given in \Pref{Icomputeassoc}, ignoring the involution, as it is central by assumption. An element from $L$ in this list has to be $x_0$; and elements from $Lj$ are $y_0j$ or $(x_0y_0)j$. All the associators in $L$, with entries from $\set{x_0,y_0,x_0y_0}$, are trivial by $L$ being diassociative. This proves $[u,v,w] = 1$ when at least two of $u,v,w$ are equal to $x_0$. Also, after the associators in $L$ vanished, it shows that
$$[a,bj,cj] = [cj,a,bj] = [bj,cj,a] = [b,a][c,a],$$
but now $a = x_0$ and $b,c \in \set{y_0, x_0y_0}$, and so $bc \in Z \cup Za$, and $[b,a][c,a] = [bc,a] = 1$ by \eq{commlin}.
Finally, we have that $[aj,bj,cj] = [ab,c][a,b]$ with $a,b,c \in \set{y_0,x_0y_0}$; if $a = b$ we are done; and if $c \in \set{a,b}$ we are done by \eq{commlincor}.
\end{proof}

\begin{cor}\label{BigD}
The variety $\var{ZA}_0 \cap \Exptwo \cap \superCentral \cap \DIAS$ is perfect.
\end{cor}
\begin{proof}
Let $\var{V} = \var{ZA}_0 \cap \Exptwo \cap \superCentral$, which was observed in \eq{cool} to be perfect. Then $(\var{V} \cap \DIAS)' = \var{V}' \cap \DIAS' = \var{V} \cap \DIAS' = (\var{V} \cap \DIAS) \cap \DIAS' = \var{V} \cap \DIAS$ by the theorem.
\end{proof}

\section{Octonion loops}\label{sec:oct}


If $L \in \var{ZA} \cap \Exptwo$, then $L/\Zent(L)$ is an elementary abelian $2$-group. The {\bf{dimension}} $\dim(L)$ of a loop $L \in \var{ZA} \cap \Exptwo$, defined as the dimension of the vector space $L/\Zent(L)$ over $\F_2$, plays a key role in the next sections. The dimension is zero when $L$ is an elementary abelian group. That the dimension cannot be~$1$ is a well-known exercise on groups, which applies to loops as well.

\begin{prop}
If $P \leq L$ are in $\var{ZA} \cap \Exptwo$ then $\dim(P) \leq \dim(L)$.
\end{prop}
\begin{proof}
We have that
$P \cap \Zent(L)\Zent(P) = (P \cap \Zent(L))\Zent(P) = \Zent(P)$, so that
$P/\Zent(P) \isom \Zent(L)P/\Zent(L)\Zent(P)$ by Noether's isomorphism theorem, which is a quotient of the subspace $\Zent(L)P/\Zent(L) \sub L/\Zent(L)$.
\end{proof}

The following example shows that a $2$-dimensional loop need not be diassociative.
\begin{exmpl}[$L \in \var{ZA} \cap \superCentral$ of dimension $2$ which is not diassociative] 
Let $L = \set{\pm 1, \pm a_1, \pm a_2, \pm a_3}$ be a commutative loop defined by $a_i^2 = -1$ and $a_ia_j = a_k$ for any permutation $\set{i,j,k}$ of $\set{1,2,3}$. The identity map is a (super-central) involution. But $L$ is not alternative, as $[a_i,a_i,a_j] = -1$.
\end{exmpl}

Motivated by \cite{ExtraII}, we say that a $3$-dimensional loop in $\var{ZA} \cap \Exptwo$ is an {\bf{octonion loop}} if it satisfies the Moufang identity (equivalently, it is extra; see \cite{KK}).

\begin{exmpl}
Let $F$ be a field of characteristic not $2$, and $i,j,k$ the standard generators of an octonion algebra over $F$. Then
$O = \mul{F}\sg{i,j,k}$ is an octonion loop (which is not associative). Indeed $O$ is Moufang as a multiplicative subloop of the alternative algebra $F[i,j,k]$, and $O/\Zent(O) \isom \Z_2^3$.
\end{exmpl}

\begin{lem}\label{assocofM}
Let $O$ be an octonion loop. Then the associator of every three generators is a fixed central element of order at most $2$.
\end{lem}
\begin{proof}
We use Bruck's lemma on equivalence of identities of Moufang loops, \cite[Lemma VII.5.5]{Bruck}. The loop $O$ clearly satisfies the identity (i) $[[x,y,z],x]=1$ of the lemma. Therefore, it satisfies the identity (iii) of the lemma, namely $[x, y, z]^{-1}= [x^{-1}, y, z] = [x,y,z]$, so $\alpha = [x,y,z]$ has order $2$.
Next, the associator is fully symmetric by (5.21) in \cite{Bruck}, and satisfies $[x,y,z]=[x,yz,z]$ by (v) of the same lemma. View $x,y,z$ as basis vectors for a $3$-dimensional vector space $V$ over $\F_2$. The automorphism group $\GL[3](\F_2)$ of this space acts transitively on the bases of~$V$, and thus on the associators $[v_1,v_2,v_3]$ for bases $\set{v_1,v_2,v_3}$. Now, the permutations and the elementary operation $1+e_{23}$  generate $\GL[3](\F_2)$, so the above identities show that the associator of any basis is equal to $\alpha$ (clearly $\alpha = 1$ if and only if $O$ is associative).
\end{proof}

Therefore, octonion loops belong to the class of loops studied in \cite{CG-1}.

\begin{cor}
An octonion loop is a central extension of an abelian group by a group of order $2$.
\end{cor}

\section{Stability under automorphisms}\label{sec:stab}

Every automorphism of $(L,*)$, which fixes $\gamma$ and~$\epsilon$, can be extended to an automorphism of $M = \CD(L,*,\gamma)$ by fixing $j$. Our goal in this section is to find conditions under which this process describes the full automorphism group of $M$.
We assume $L$ belongs to the perfect variety $\var{ZA}_0 \cap \Exptwo \cap \superCentral$ of \eq{cool}.

\begin{prop}
Let $L \in \var{ZA}_0 \cap \Exptwo \cap \superCentral$ and $M = \CD(L,*,\gamma)$.
Then
\begin{equation}\label{dimdim}
\dim(M) \geq \dim(L),
\end{equation}
with equality if and only if the involution on $L$ is the identity.
\end{prop}
\begin{proof}
As vector spaces, $\dim(L/\Zent(L)) \leq \dim(L/\Zent(L,*)) = \dim(L/ L \cap \Zent(M)) = \dim(L\Zent(M)/\Zent(M)) \leq \dim(M/\Zent(M))$ with equality in the first step, as well as in the final step, if and only if the involution on~$L$ is identity by \Cref{ZentMnot1} and \Pref{ZentM1}.
\end{proof}


\subsection{The generator $j$ induces Moufang}

Let $V$ be a vector space over~$\F_2$, and $j \in V$ a fixed nonzero element. The mapping $U \mapsto \sg{U,j}$ takes $k$-dimensional subspaces disjoint from $j$ to $(k+1)$-dimensional subspaces containing $j$. We say that subspaces $U,U'$, disjoint from $j$, are {\bf{$j$-partners}} if $U \neq U'$ and $\sg{U',j} = \sg{U,j}$.  Every pair of $3$-dimensional $j$-partners can be represented in the form $U = \sg{a,b,c}$ and $U' = \sg{a,b,cj}$ (writing $V$ multiplicatively). Every element of $\sg{U,j}$, except for~$j$, belongs to some $j$-partner of $U$.

We will apply this notion to a Cayley-Dickson double $M = \CD(L,*) = L \cup Lj$, by projecting subloops of $M$ to their images in $M/\Zent(M)$. Here we note that out of every partnership class, precisely one subloop is contained in~$L$; namely, $\sg{U,j} \cap L$.
\begin{lem}\label{Full1}
Let $L \in \var{ZA} \cap \Exptwo \cap \Central$. Let $M = \CD(L,*,\gamma)$ be its double. Let $O \leq L$ be a $3$-dimensional subloop. If $O$ and some $j$-partner of $O$ are Moufang, then $O$ is associative.
\end{lem}
\begin{proof}
Write $O = \sg{a,b,c}$ where the Moufang $j$-partner is $\sg{a,b,cj}$. By the identities given in \Pref{Icomputeassoc}, and the involution being central, we have that
$[a,b,cj] = 
[a,b,c][b,a]$
and $[a,cj,b] = 
[b,a,c]^{-1} [a,b,c] [b,a]$. If $\sg{a,b,cj}$ is Moufang then from \Lref{assocofM} we find that $[b,a,c] = [a,b,cj][a,cj,b]^{-1} = 1$, and we conclude by Moufang's theorem that $O$ is associative.
\end{proof}

\begin{lem}\label{Full2}
Let $L \in \var{ZA}_0 \cap \Exptwo \cap \superCentral$ and $M = \CD(L,*,\gamma)$.
Then $L$ is diassociative
if and only if $\sg{Q,j}$ is Moufang for every $2$-dimensional subloop $Q \leq L$.
\end{lem}
\begin{proof}
$(\Rightarrow)$ Let $Q$ be a $2$-dimensional subloop of $L$. Since $L$ is diassociative, $Q$ is a group, so since~$\gamma$ is central, $\sg{Q,\gamma}$ is a group as well. By \eq{coolrem0succ} the involution on $L$ is normal, so $\sg{Q,j} = \CD(\sg{Q,\gamma},*,\gamma)$ is Moufang by \Cref{fin1}. 
$(\Leftarrow)$ Let $Q$ be a subloop of $L$ generated by two elements, then $\sg{Q,j}$ is Moufang and thus diassociative, showing that~$Q$ is associative.
\end{proof}

We say that an element $x$ is {\bf{locally Moufang}} in a loop $M \in \Exptwo$ if every $3$-dimensional subloop of $M$ containing $x$ is Moufang. Clearly, a necessary condition for the existence of a locally Moufang element is that $M$ is diassociative.
Since every $3$-dimensional subloop $O \leq \CD(L,*,\gamma)$ containing~$j$ has the form $O = \sg{O \cap L, j}$, we proved:
\begin{cor}\label{Full2+}
For $L \in \var{ZA}_0 \cap \Exptwo \cap \superCentral \cap \DIAS$, the element $j$ is locally Moufang in $M = \CD(L)$.
\end{cor}

\subsection{Two step doubling}

We now consider a chain $T \subset L \subset M$, where $L = \CD(T) = T \cup Tj$ and $M = \CD(L) = L \cup Lj'$.
Recall from \Pref{muinduces} that a central involution on a loop $T$ induces a homomorphism of the center. This homomorphism is trivial precisely when $\Zent(T) = \Zent(T,*)$.


We now prove a converse to \Cref{Full2+}.
\begin{lem}\label{Full3}
Assume $T \in \var{ZA}_0 \cap \Exptwo \cap \superCentral$ is a noncommutative loop such that $\Zent(T) = \Zent(T,*)$. 
Let $T \sub L \sub M$ as above. (In particular $M$ is not Moufang).

If $x \in M$ is locally Moufang, then $x \in \Zent(T) \cup \Zent(T) j' \cup (\KZent(T) -\Zent(T)) (j j')$,
where $\KZent(T) - \Zent(T)$ is the set difference.
\end{lem}
\begin{proof}
First assume $x \in L$. Let $y,z \in L$ be such that $\sg{x,y,z}$ is $3$-dimensional. Then $\sg{x,y,z}$ and $\sg{x,y,zj'}$ are Moufang by assumption, so $\sg{x,y,z}$ is associative by \Lref{Full1}.  This proves that $x \in \Nuc{L}$, so $x \in \Zent(T)$ by \Pref{nucMall}.

Now assume $x \in Lj'$, so write $x = x'j'$ where $x' \in L$. Take $y, z \in L$, such that $\sg{x', y, z} $ is of dimension $3$. Then $[z, y, x'j'] = [z, y, x'][y, z]$ and $[z, x'j', y] = [y, z, x']^{-1}[z, y, x'][y, z]$.
Since $\sg{y,z,x'j'} = \sg{x,y,z}$ is $3$-dimensional Moufang, these associators must be equal. Hence $[y, z, x'] = 1$. This proves $x' \in \Nuc[r]{L}$.

Now, if $x' \in T$, then $x' \in \Zent(T)$ by \Pref{nucM}. So assume
$x' \in Tj$, so write $x' = aj$. 
By \Pref{nucM}.(1), this already proves that $a \in \KZent(T)$, and it remains to show that $a \not \in \Zent(T)$, but the condition $(ab)c = a(cb)$ in \Pref{nucM}.(1) does not hold for $a \in \Zent(T)$ if we choose noncommuting $b,c \in T$.
\qedhere
\end{proof}

With a stronger assumption on $\KZent(T)$, we get a more concise statement:
\begin{cor}\label{Full3.0}
Assume $T \in \var{ZA}_0 \cap \Exptwo \cap \superCentral$ is noncommutative, and $\KZent(T) = \Zent(T,*)$.
Let $T \sub L \sub M$ be as above.

Then any locally Moufang element $x \in M$ is in $\Zent(T)\sg{j'}$.
\end{cor}

\medskip

We say that a loop~$T$ with involution is {\bf{admissible}} if $$T \in \var{ZA} \cap \Exptwo \cap \Central \cap \DIAS,$$ and $\KZent(T) = \Zent(T,*)$. It follows that $T \in \var{ZA}_0 \cap \superCentral$, so $T$ belongs to the perfect variety of \Cref{BigD}. By \Pref{thecenter}, if the involution on $T$ is nonidentity, then the property $\KZent(T) = \Zent(T,*)$ is also preserved in the doubling process.

\begin{cor}\label{sofar}
Let $T$ be an admissible loop which is not an abelian group. Then any automorphism of $M = \CD^2(T) = \sg{T,j,j'}$ stabilizes $j'$ modulo the center.
\end{cor}
\begin{proof}
Let $L = \CD(T,*)$. By \Cref{BigD}, $L \in \var{ZA}_0 \cap \Exptwo \cap \superCentral$, so we can apply \Cref{Full2+} to~$L$, and obtain that every $3$-dimensional subloop of~$M$ containing~$j'$ is Moufang. On the other hand, by \Lref{Full3}, $j'$ is the only element (up to center) satisfying this property, so any automorphism must preserve~$j'$ up to the center.
\end{proof}

\begin{exmpl}
\Cref{sofar} applies to the quaternion group $T = Q_2$ (with the symplectic involution induced from the quaternion algebra), proving that any automorphism of $Q_4 = \sg{Q_3,j'}$ stabilizes $j'$ modulo the center $\set{\pm 1}$. But it does not apply to the abelian group $Q_1$ (although $Q_1 \in \var{ZA}_0 \cap \Exptwo \cap \superCentral \cap \DIAS$), and indeed the automorphism group of $Q_3$ acts transitively on $Q_3/\Zent(Q_3) \isom \Z_2^3$.
\end{exmpl}

\section{The automorphism group}\label{sec:12}

We can finally prove the main result of this part of the paper. Recall that anti-commutative loops were defined in \Ssref{ASAC}. Following the terminology for groups, we say that $A \leq L$ is a {\bf{characteristic subloop}} if any automorphism of $L$ preserves $A$.

\begin{thm}\label{mainM}
Let $T$ be an admissible anti-commutative 
loop, which is not an abelian group.
Then $L = \CD(T) = \sg{T,j}$ is a characteristic subloop of its double $M = \CD(L) = \sg{T,j,j'}$.
\end{thm}
\begin{proof}
First note that by \Cref{ZentMnot1}, $\Zent(M) = \Zent(L) = \Zent(T,*)$, which we denote by~$Z$. Since $\dim(T/Z) \geq 2$, we have that $\dim(M/Z) \geq 4$.
Let $\phi$ be an automorphism of $M$. Let $s \in L$ and assume, by contradiction, that $\phi(s) \not \in L$. Write $\phi(s) = cj'$ for $c \in L$.

Let $P$ be a subloop of $M$ containing $P_0 = \sg{\phi^{-1}(j), j, s, j'}Z$ such that $\dim(P/Z) = 4$. Let $P' = P \cap L$; then $P/P' \isom PL/L = M/L$ so $\dim(P'/Z) = 3$, and $P'$ is Moufang by \Cref{Full2+} since $j \in P'$. Therefore $\phi(P') \supseteq \phi(P_0) \cap L = \sg{j, \phi(j), c,j'}Z  \cap \phi(L)$ is Moufang.

We now use the fact that $\phi(j') \equiv j' \pmod{Z}$ by \Cref{sofar}.
Since $\sg{\phi(P'), j'} = \phi(\sg{P',j'}) = \phi(P) = \sg{\phi(P) \cap L, j'}$, $\phi(P')$ is a $j'$-partner of $\phi(P) \cap L$, so the latter is associative by \Lref{Full1}.
Since $j \in \phi(P) \cap L$ by construction, and $\phi(P) \cap L = \sg{\phi(P) \cap T, j}$, $\phi(P) \cap T$ must be an abelian group by \Eref{der1}. But this contradicts the assumption that $T$ is anti-commutative since $\dim(\phi(P) \cap T) = 2$.
\end{proof}

Recall the notation from \Ssref{ss:Aut0}. Together with \Cref{sofar}, \Tref{mainM} proves for $M$ as in the theorem that
\begin{equation}\label{goal}
\Aut(M) = \Aut(M;\,L,Zj')
\end{equation}
and thus
\begin{equation}\label{goal*}
\Aut(M,*) = \Aut(M,*;\,L,Zj').
\end{equation}
These groups were shown in \Pref{Aut0} to be subgroups of the semidirect product $\Aut(L,*) \semidirect \Zent(L)$.

The following chain of loops motivated this paper.
\begin{exmpl}\label{Qn}
Let $Q_n$ denote the loop of the standard basis elements of the classical Cayley-Dickson algebra $A_n$ over a field of characteristic not~$2$ ($\dim A_n = 2^n$). Namely $Q_n = \CD(Q_{n-1},*,-1,-1)$, where $Q_0 = \set{\pm 1}$ with the trivial involution. Thus $Q_1 = \sg{i} \isom \Z_4$ with the inverse involution, and~$Q_2$ is the quaternion group, again with the inverse involution.
\end{exmpl}

\begin{rem}\label{Kirsh}
Lemma 36 in \cite{Kirsh} claims that if $\sg{x,y,z} \leq Q_{n-1}$ is an octonion subloop (where $n \geq 4$), then $\sg{xj,yj,zj}$ is a nonassociative octonion subloop as well. However, in this loop $[xz,yz,zj] = 1$, which can be confirmed directly or by \cite[Lemma 34(a)]{Kirsh} (equivalently by our \Pref{Icomputeassoc});
however, the associators of bases of a nonassociative octonion loop are all equal to $-1$. The claim is thus incorrect.
\end{rem}

The loops in \Exref{Qn} generate the standard Cayley-Dickson algebras. Let us now consider the general case.
\begin{exmpl}\label{genQn}
Let $F$ be a field of characteristic not $2$. Let $\set{\gamma_0,\gamma_1,\dots} \sub \mul{F}$ be scalars.

Denote by $Q_0^{\circ}$ the multiplicative group $\mul{F}$ with the trivial involution. For $i = 1,\dots$, let $$Q_i^{\circ} = \CD(Q_{i-1}^{\circ},*,\gamma_{i-1},-1) = \sg{Q_{i-1}^{\circ},j_{i-1}}.$$
The common center is $\mul{F}$, and $Q_i^{\circ}/\mul{F} \isom \Z_2^{i}$. (When all $\gamma_i = -1$, $Q_i^{\circ}$ is the central product $\mul{F}Q_i$, namely the Cartesian product $\mul{F}\times Q_i$ modulo identifying $-1$ in both groups).

To emphasize the role of the scalars, let us denote $Q_n^\circ$ by $(\gamma_0,\dots,\gamma_{n-1})^\ell$.

Identifying the center with the multiplicative group of the base field, the loop algebra $F(\gamma_0,\dots,\gamma_{n-1})^\ell$ is the Cayley-Dickson algebra $(\gamma_0,\dots,\gamma_{n-1})_F$, of dimension $2^n$. Thus $F(\gamma_0)^\ell = F[\sqrt{\gamma_0}]$ is a (possibly split) quadratic extension, $F(\gamma_0,\gamma_{1})^\ell$ is the quaternion algebra $(\gamma_0,\gamma_1)_F$, and $F(\gamma_0,\gamma_1,\gamma_2)^\ell$ is the octonion algebra $(\gamma_0,\gamma_1,\gamma_2)_F$.
\end{exmpl}

\begin{exmpl}
Continuing \Exref{genQn}, if $n\geq 2$ then $(\gamma_0,\dots,\gamma_{n-1})^\ell$ is a loop satisfying all the conditions of \Tref{mainM}. Therefore, every automorphism of $(\gamma_0,\dots,\gamma_{n-1},\gamma_n,\gamma_{n+1})^\ell$ fixes the newest generator $j_{n+1}$ modulo the center. In particular the automorphism group $\Aut((\gamma_0,\dots,\gamma_{n+1})^\ell)$ is a subgroup of $\Aut((\gamma_0,\dots,\gamma_{n})^\ell,*) \semidirect \mul{F}$.
\end{exmpl}

Note that for an anti-symmetric admissible loop, the involution cannot be the identity unless $L$ is an abelian group.
\begin{thm}\label{FINAL}
Let $T$ be an admissible loop which is anti-symmetric and anti-commutative, but not an abelian group. Let $Z = \Zent(T)$.
Let $n \geq 2$, and let
$$\gamma_0,\dots,\gamma_{n-1}, \epsilon_0, \dots,\epsilon_{n-1} \in Z$$ be central elements with $\epsilon_i \neq 1$ and $\epsilon_i^2 = 1$ ($i=1,\dots,n-1$). Let
$$M = \CD^n(L,*; \gamma_0,\ldots,\gamma_{n-1}; \epsilon_0,\dots,\epsilon_{n-1})$$
be the Cayley-Dickson doubling, repeated $n$ times.

Then $\Aut(M)$ is a subgroup of the semidirect product $$\Aut(T_1,*) \semidirect Z^{n-1},$$
where $T_1 = \CD(T,*,\gamma_0,\epsilon_0)$, composed of the vectors $(\s;p_1,\dots,p_{n-1})$ for which $\s(\gamma_i)\gamma_i^{-1} = p_i^2$ for $i \leq n-1$ and $\s(\epsilon_i) = \epsilon_i$ for $i \leq n-2$.
\end{thm}
\begin{proof}
Denote $T_0 = T$ and for $i = 1,\dots,n$, let
$$T_i = \CD(T_{i-1},*,\gamma_{i-1},\epsilon_{i-1}) = \sg{T_{i-1},j_i},$$
when $j_i$ is the standard generator of the doubling at this stage, with $j_i^2 = \gamma_i$.
Thus $M = T_n$. We claim that every~$T_i$ is admissible, anti-symmetric and anti-commutative. Since $T_0$ is not an abelian group, the involution on $T_0$ is nonidentity, so this holds for any $T_i$.
Being in $\var{ZA}_0 \cap \Exptwo \cap \superCentral \cap \DIAS$ passes to the double by \Cref{BigD}. Anti-symmetry and anti-commutativity pass to the double when the involution on the given loop is nonidentity and $\epsilon \neq 1$, by \Pref{AS'} and \Pref{AC'}. And finally,
it follows from \Cref{ZentMnot1} that every central element is symmetric.

Now, for $i = 2,\dots,n$, we can take $T$ in~\Tref{mainM} to be $T_{i-2}$, which gives $\Aut(T_{i}) = \Aut(T_i;\, T_{i-1}, Z j_i)$. Combining these claims, we find that every automorphism of $M = T_n$ preserves the flag $T_1 \sub T_2 \sub \cdots \sub T_n$, and is determined by the restriction of the automorphism to $T_1$ and by the image of each $j_i$ in $Z j_i$. The explicit conditions on $(\s;p_1,\dots,p_n)$ are given in~\Pref{Aut0}, which in particular shows that $\s \in \Aut(T_i)$ fixes~$\epsilon_{i-1}$.
\end{proof}

Recall the loops $Q_n$ from \Exref{Qn}. Then $T = Q_2$ is the quaternion group, with the inverse involution, which is admissible. The joint center is $Z = \set{\pm 1}$, and the action of $\Aut(Q_3)\isom \GL[3](\F_2)$ on the center is trivial. \Tref{FINAL} applies, and we obtain:
\begin{cor}
For any $n \geq 3$,
$$\Aut(Q_n) = \Aut(Q_3,*) \times \set{\pm 1}^{n-3}.$$
\end{cor}

In the context of the Cayley-Dickson algebras, or their basis loops, it makes sense to consider automorphisms fixing the base field~$F$.
Working with the loops $(\gamma_0,\dots,\gamma_{n-1})^\ell$ of \Exref{genQn}, the same theorem proves the following:
\begin{cor}\label{final}
For any $n \geq 3$, and any choice of the defining scalars $\gamma_0,\dots,\gamma_{n-1} \in \mul{F}$,
$$\Aut_F((\gamma_0,\dots,\gamma_{n-1})^\ell) = \Aut_F((\gamma_0,\gamma_1,\gamma_{2})^\ell,*) \times \set{\pm 1}^{n-3}.$$
\end{cor}

Again considering only isomorphisms fixing the base field $F$, the same arguments prove a loop counterpart to Schafer's theorem on Cayley-Dickson algebras \cite{Sch0}:
\begin{thm}
The loops $(\gamma_0,\dots,\gamma_{n-1})^{\ell}$ and $(\gamma'_0,\dots,\gamma'_{n-1})^{\ell}$ are isomorphic as extensions of $\mul{F}$ if and only if $(\gamma_0,\gamma_1,\gamma_{2})^{\ell} \isom (\gamma'_0,\gamma_1',\gamma'_{2})^{\ell}$ and $\gamma_i' \in \mul{F}^2\gamma_i$ for $i = 4,\dots,n$.
\end{thm}

\Cref{final} is not always enlightening, because for example when the~$\gamma_i$ are independent in the $\Z_2$-module $\mul{F}/\mul{F}^2$, a-priori $\Aut_F(Q_n^{\circ}) = \set{\pm 1}^n$ because an automorphism~$\s$ must fix every $j_i^2 = \gamma_i$ as a scalar, forcing $\s(j_i) = \pm j_i$. However, the theorem applies also to the generic loop. Let $k$ be any field of characteristic not $2$. Let $F = k(\gamma_0,\dots,\gamma_{n-1})$ be the function field in $n$ variables over~$k$. Construct the loops~$Q_n^{\circ}$ of \Exref{genQn}.
Then $$\Aut_k(Q_n^{\circ}) = \Aut_k(Q_3^{\circ},*) \times (\mul{F})^{n-3},$$
emphasizing the fact that although $\Aut_k(Q_n^{\circ})$ acts transitively on $\set{\gamma_0,\gamma_1,\gamma_2}$, the $\gamma_i$'s with $i \geq 3$ are fixed modulo squares.

\end{document}